\documentclass[a4paper,psamsfonts,reqno]{amsart}
\usepackage{amssymb,amscd,amsmath}

\usepackage{upref}
\usepackage[american]{babel}
\usepackage[all]{xy}
\usepackage{fancyhdr}
\usepackage{enumerate}

\theoremstyle{plain}
\newtheorem{theo}{Theorem}[section]

\newtheorem{prop}[theo]{Proposition}
\newtheorem{lem}[theo]{Lemma}
\newtheorem{coro}[theo]{Corollary}

 \theoremstyle{definition}

 \newtheorem{ex}[theo]{Example}
 \newtheorem{exs}[theo]{Examples}

 \theoremstyle{remark}

 \newtheorem{rem}[theo]{Remark}
  \newtheorem{rems}[theo]{Remarks}

\numberwithin{equation}{section}

\newcommand{\freealgebra}[2]{{#1} \langle {#2}\rangle}

\newcommand{\naturals}{\mathbb{N}}
\newcommand{\integers}{\mathbb{Z}}

\DeclareMathOperator{\supp}{supp} \DeclareMathOperator{\h}{h}
 \DeclareMathOperator{\Aut}{Aut}
\DeclareMathOperator{\gr}{gr}
 \DeclareMathOperator{\Der}{Der}

\begin{document}

\title {The Inversion Height of the Free Field is Infinite}

\author{Dolors Herbera}
\address{Departament de Matem\`atiques \\
Universitat Aut\`onoma de Barcelona \\ 08193 Bellaterra (Barcelona),
Spain} \email{dolors@mat.uab.cat}

\author{Javier S\'anchez }
\address{
Department of Mathematics - IME \\ University of S\~ao Paulo\\ Caixa
Postal 66281\\ S\~ao Paulo, SP\\ 05314-970, Brazil}
\email{jsanchez@ime.usp.br}

\thanks {The second named author was supported by  Funda\c{c}\~ao
de Amparo \`a Pesquisa do Estado de S\~ao Paulo (FAPESP) processo
n\'umero 2009/50886-0. \\ Both authors acknowledge partial support
from  DGI MICIIN MTM2011-28992-C02-01, and by the Comissionat per
Universitats i Recerca of the Generalitat de Ca\-ta\-lunya, Project
2009 SGR 1389. \protect\newline 2000 Mathematics Subject
Classification. Primary: 16K40; 16S34; 16S35 Secondary: 16S10;
16W60.}

\maketitle
\begin{abstract}
Let $X$ be  a finite set with at least two elements, and let $k$ be
any commutative field. We prove that the inversion height of the
embedding $k\langle X\rangle \hookrightarrow D$, where $D$ denotes
the universal (skew) field of fractions of the free algebra
$k\langle X\rangle$, is infinite. Therefore, if $H$ denotes the free
group on $X$, the inversion height of the embedding of the group
algebra $k[H]$ into the Malcev-Neumann series ring is also infinite.
This answer in the affirmative a question posed by Neumann in 1949
\cite[p.~215]{Neumann}.

We also give an infinite family of examples of non-isomorphic fields
of fractions of $k\langle X\rangle$ with infinite inversion height.

We show that the universal field of fractions of a crossed product
of a commutative field by the universal enveloping algebra of a free
Lie algebra is a field of fractions constructed by Cohn (and later
by Lichtman). This extends a result by A. Lichtman.
\end{abstract}

\tableofcontents

\section{Introduction}

Let $X$ be a set with $\vert X\vert\geq 2$, $H$ be the free group on
 $X$ and $k$ be a commutative field. Choose a total order on $H$ such that $(H,<)$
 is an ordered group.  Consider the Malcev-Neumann series ring
 $k((H,<))$ associated with the group ring $k[H]$. B.H. Neumann conjectured in
 \cite[p.~215]{Neumann} that

\medskip

\begin{enumerate}[\ \  (N)]
\item the inversion height of the embedding
$k[H]\hookrightarrow k((H,<))$ is infinite.  Equivalently,  in the
(skew) subfield $E=E(X)$ of $k((H,<))$ generated by $k[H]$ there
exist elements which need an arbitrary large number of nested
inversions to be constructed as a rational expression from elements
of $k[H]$.
 \end{enumerate}

\medskip

The  field $E=E(X)$  can be characterized by its categorical
properties. It was proved by Lewin \cite{Lewin} that it is the
universal field of fractions of $k[H]$ and, hence, it is also the
universal field of fractions of  $k\langle X \rangle$, the free
algebra on $X$; because of that $E$ is usually named \emph{the free
field on $X$}. We recall that $k\langle X\rangle$ can also be seen
as the enveloping algebra of the free Lie algebra on $X$.

The interest on conjecture (N) was renewed in
\cite{Gelfandretakhwilson} where the theory of Quasideterminants was
developed. C. Reutenauer brilliantly proved in
\cite[Theorem~2.1]{Reutenauerinversionheight} that  the conjecture
holds when $X$ is infinite and  $k$ is a commutative field. As
suggested in \cite[Section~5.2]{Reutenauerinversionheight}, it was
expected that (N) should hold in general  because a free algebra $R$
over a set of at least two elements contains many subalgebras $S$
that are isomorphic to a free algebra over an infinite (countable)
set. The difficulty in settling the  question with this approach was
being able to choose a subalgebra $S$ such that the universal field
of fractions of $S$ can be seen inside the one of $R=K\langle
X\rangle$ and that, in addition, the inversion height is preserved
through the embedding. In this paper, we overcome this problem
considering the more flexible structure of crossed product. More
precisely, seeing $R$   as a crossed product of the subalgebra $S$
with \emph{something else}   we can produce, via Reuteneauer's
result, elements in $E$ of arbitrary inversion height. Hence we give
the final step to solve conjecture (N).

Crossed products can be considered in the group context,
in the context of  Lie algebras or, unifying both settings,
for Hopf algebras. They have proved to be specially suitable
for induction-type arguments and also in the
construction of quantum deformation of classical algebraic objects.

Throughout  the paper, we give several constructions of elements in
the free field $E$ of arbitrary inversion height, keeping in
parallel the point of view of crossed products of Lie algebras and
the one of group crossed products. In Section \ref{sec:elementary},
we give the most elementary constructions to produce  elements of
arbitrary large inversion height. We use the ideas of an embedding
due to Cohn \cite{Cohnembeddingtheorem} that allows to see the free
algebra as an Ore differential extension of a free algebra on
infinitely many variables. Such kind of extensions are the easiest
example of crossed product of Lie algebras. Then we are able to give
an elementary solution to conjecture (N) in
Theorem~\ref{theo:firstsolution}.

 On the group side, if $H$ is a
free group, any onto group homomorphism  $\varphi$ from $H$ to an
infinite cyclic group allows to see $k[H]$ as a skew Laurent
polynomial ring with coefficients on the group algebra over the free
group $\mathrm{Ker} \, \varphi$, again this is the easiest example
of crossed product of groups.  Such description of the group algebra
allows us to give in Theorem \ref{theo:solutiontoNeumann} another
elementary solution to conjecture (N).

In Section \ref{sec:other}, we deeply use  the theory of crossed
product of groups to produce infinitely many non-isomorphic
embeddings of the free algebra into division rings of infinite
inversion height. Hence, the property of having infinite inversion
height does not characterize the universal field of fractions.

In Sections \ref{sec:crossedproducts} and
\ref{sec:afieldoffractions}, we develop some specific theory of
crossed products for Lie algebras, and we give a construction of a
field of fractions, as a subfield of a power series ring, for the
crossed product of a field by a residually nilpotent Lie algebra
with a $Q$-basis. In the case of a free Lie algebra $H$ or, more
generally, when the crossed product is a fir, this gives a
construction of the universal field of fractions.  In
Section~\ref{sec:furtherexamples}, we use this theory to produce
further examples of elements with arbitrary large inversion height
into the free field. A different line of applications of this
construction is given in Example~\ref{ex:poisson} to the enveloping
algebra of the free Poisson field, cf.
\cite{Makar-LimanovUmirbaevfreePoissonfields}.

In the case of an ordered  group, the  Malcev-Neumann series ring
gives a very neat way to embed a crossed product of an arbitrary
field by the group  into a field. As mentioned before, when the
group is free, this yields an embedding of the universal field of
fractions of the crossed product in such power series ring. This was
proved by Lewin in \cite{Lewin} using a deep result of Hughes on the
uniqueness of some field of fractions \cite{Hughes}.

On the Lie algebra side, a well known result of Cohn implies that
any crossed product of a field by the universal enveloping algebra
of a Lie algebra can be embedded into a field, cf.
Proposition~\ref{prop:canonicalfieldoffractions}, which we call the
canonical field of fractions. But an analog of the Malcev-Neumann
series ring construction, possibly containing the canonical field of
fractions, is missing in the setting of \emph{ordered Lie algebras}.
Our main results in Section 6 aim  to fill this gap in the case of
crossed products of residually nilpotent Lie algebras with a
$Q$-basis. In our constructions, we follow and extend results and
ideas due to Lichtman \cite{Lichtmanvaluationmethods,
Lichtmanuniversalfields}.

As we have already mentioned, all our results on inversion height
are based on Reutenauer's ones. It seems an interesting and
challenging question to extend Reutenauer's results  from
commutative fields to arbitrary (skew) fields. We note that our
approach to pass from the case of countable infinitely many
variables to the finite one does not use any commutativity and it
works for general crossed products.

\section{Preliminaries}

We begin this section fixing some notions that will be used
throughout the paper.

All rings are assumed to be associative and with $1$. A morphism of
rings  $\alpha\colon R\rightarrow S$ always preserves $1$'s, i.e.
$\alpha$ sends $1_R$ to $1_S$.

By an \emph{embedding} $\iota\colon R\hookrightarrow E$ we mean an
injective morphism of rings where we identify $R$ with its image in
$E$.

A \emph{domain} is a nonzero ring $R$ such that the product of any
two nonzero elements is nonzero.

Following \cite{Cohnskew}, a \emph{field} $E$ is a nonzero ring such
that every nonzero element has an inverse, i.e. if $x\in
E\setminus\{0\}$ there exists $x^{-1}\in E$ such that
$xx^{-1}=x^{-1}x=1$.

Note that  domains and fields are not supposed to be commutative. In
the literature, our concept of field is also known as division ring
or skew field.

\subsection{Skew polynomial rings and skew Laurent series} \label{sec:1.1}
Let $S$ be a ring and $\alpha\colon S\rightarrow S$ an injective
endomorphism of rings.

A (left) \emph{$\alpha$-derivation} is an additive map $\delta\colon
S\rightarrow S$ such that
$\delta(ab)=\delta(a)b+\alpha(a)\delta(b)$.

We denote by $S[x;\alpha,\delta]$ the \emph{skew polynomial ring}.
It is a ring extension of $S$ which is a free left $S$-module with
basis $\{1,x,\dotsc,x^n,\dotsc\}$, thus the elements can be uniquely
written as $$a_0+a_1x+\dotsb+a_nx^n \textrm{ with } a_i\in S,\
n\in\naturals,\ a_n\neq0,$$ and $xa=\alpha(a)x+\delta(a)$ for all
$a\in S$. When $\delta=0$, we write $S[x;\alpha]$ instead of
$S[x;\alpha,0]$, and when $\alpha$ is the identity on $S$, we write
$S[x;\delta]$ instead of $S[x;\alpha,\delta]$.

If $S$ is a domain, the ring $S[x;\alpha,\delta]$ is also a domain.
If $S$ is a left Ore domain, then $S[x;\alpha,\delta]$ is a left Ore
domain. If $S$ is a field, we denote its left Ore field of fractions
by $S(x;\alpha,\delta)$ (respectively $S(x;\alpha)$, $S(x;\delta)$).

When $\delta=0$, we can consider the skew series ring
$S[[x;\alpha]]$ which consists of all infinite series
$$a_0+a_1x+\dotsb+a_nx^n+\dotsb,\ a_n\in S\textrm{ for all } n\in\naturals,$$
with componentwise addition and multiplication based on the
commutation rule $$xa=\alpha(a)x,\ \textrm{for all } a\in S.$$ The
set $\{1,x,\dotsc,x^n,\dotsc\}$ is a left Ore set in
$S[[x;\alpha]]$, and we denote its Ore localization by
$S((x;\alpha))$. The elements of $S((x;\alpha))$ are of the form
$$x^{-r}\sum_{n=0}^\infty a_nx^n \textrm{ with } r\in\naturals,\ a_n\in S \textrm{ for all } n.$$
If $S$ is a field, $S((x;\alpha))$ is a field that contains
$S(x;\alpha)$. If $\alpha$ is bijective, the elements of
$S((x;\alpha))$ can be written as $\sum\limits_{n\geq -r}a_nx^n$
with $r\in\naturals$ and $a_n\in S$ for all $n$.

When $\delta=0$ and $\alpha$ is bijective, the subring of
$S((x;\alpha))$ consisting of the polynomials of the form
$$a_{-m}x^{-m}+a_{-m+1}x^{-m+1}+\dotsb+ a_0+a_1x+\dotsb+a_nx^n,\
\textrm{ with } a_i\in S,\ m,n\in\naturals,$$ is called
the \emph{skew Laurent polynomial ring} and denoted by
$S[x,x^{-1};\alpha]$. If $S$ is a left Ore domain,
$S[x,x^{-1};\alpha]$ is also a left Ore domain. If $S$ is a field,
the left Ore field of fractions is $S(x;\alpha)$.

\medskip

When $\delta\neq 0$ and $\alpha$ is injective, we can also construct
a similar ring of series (to understand its definition, notice that
the relation $xa=\alpha(a)x+\delta (a)$ implies that
$ax^{-1}=x^{-1}\alpha(a)+x^{-1}\delta(a)x^{-1}$). We introduce a new
variable $y=x^{-1}$, and  we consider the ring of series
$$a_0+ya_1+\dotsb+y^na_n+\dotsb \textrm{ with } a_n\in S \textrm{
for all } n\in\naturals,$$ (coefficients on the right) with
componentwise addition and multiplication based on the commutation
rule
\begin{equation}\label{eq:commutationruleseries}
ay=y\alpha(a)+y^2\alpha(\delta(a))+\dotsb+y^n\alpha(\delta^{n-1}(a))+\dotsb=\sum_{n\geq1}
y^n\alpha(\delta^{n-1}(a)),
\end{equation}
for each $a\in S$. This ring of series will be denoted by
$S[[y;\alpha,\delta]]$. The set $\{1,y,\dotsc,y^n,\dotsc\}$ is a
right Ore set and we denote by $S((y;\alpha,\delta))$ its Ore
localization. So the elements of $S((y;\alpha,\delta))$ are of the
form $$\Big(\sum_{n=0}^\infty y^na_n\Big)y^{-r} \textrm{ with }
r\in\naturals,\ a_n\in S \textrm{ for all } n\in\naturals.$$ If $S$
is a field, then $S((y;\alpha,\delta))$ is a field.

From \ref{eq:commutationruleseries}, it is easy to see that the
assignment $x\mapsto y^{-1}$ induces an injective morphism of rings
$S[x;\alpha,\delta]\to S[[y;\alpha,\delta]]$ which is the identity
on $S$. The universal property of the Ore localization, allows to
extend this embedding to an embedding of fields
$S(x;\alpha,\delta)\to S((y;\alpha,\delta)).$

Finally, we observe that if $\alpha$ is an automorphism then the
elements of $S((y;\alpha,\delta))$ can be written, in a unique way,
in the form
$$\sum_{n\ge l}^\infty  a_ny^{n} \textrm{ with }
l\in\mathbb{Z},\ a_n\in S \textrm{ for all } n\in\naturals.$$

\subsection{Crossed products and Malcev-Neumann
series}\label{subsec:crossedproducts}

Let $R$ be a ring, and let  $G$ be a group. We define a
\emph{crossed product} $RG$ (of $R$ by $G$) as an associative ring
which contains $R$ constructed in the following way. It is a free
left $R$-module   with basis $\overline G$, a copy (as a set) of
$G$. The elements in $RG$ are uniquely written as $\sum\limits_{x\in
G}a_x\bar x$ where only a finite number of $a_x\in R$ are nonzero.
Multiplication
is determined by the two rules below:\\
\indent \emph{Twisting}\index{twisting}. For $x, y \in G$
\begin{displaymath} \bar x\bar y=\tau(x,y)\overline {xy}
\end{displaymath}
 where $\tau \colon G\times G\longrightarrow
R^\times$ and  $R^\times$ denotes the group of units of $R$.\\
\indent \emph{Action}\index{action}. For $x\in G$ and $r\in R$
\begin{displaymath} \bar xr={}^{\sigma(x)}r\bar x
\end{displaymath}
 where $\sigma \colon G\rightarrow \Aut (R)$,
 $\Aut(R)$ denotes the group of automorphisms of $R$ and  ${}^{\sigma(x)}r$ denotes the image of $r$ by $\sigma(x)$.
Hence if $\sum_{x\in G}a_x\bar{x}, \sum_{x\in G}b_x\bar{x}\in RG$,
then
\begin{equation}\label{eq:product}
\sum_{x\in G}\Big(\sum_{yz=x}a_y{}^{\sigma(y)}b_z
\tau(y,z)\Big)\bar{x}.
\end{equation}

We stress that neither $\sigma$ nor $\tau$ need to preserve any kind
of structure.

 If $H$ is a subgroup of $G$, then
$ RH=\{\eta \in RG\mid \supp \eta\subseteq H\}$ is the naturally
embedded sub-crossed product.

Crossed products  do not have a natural basis. If \mbox{$d\colon
G\rightarrow R^\times$} assigns to each element $x\in G$ a unit
$d_x$, then \mbox{$\widetilde{G}=\{\tilde{x}=d_x\bar{x}\mid x\in
G\}$} is another $R$-basis for $RG$ which still exhibits the basic
crossed product. After a change of basis if necessary, we will
always suppose that $1_{RG}=\bar{1}$.

A crucial property of crossed products  is the following. If $N$ is
a normal subgroup of $G$ then $RG=RN\frac{G}{N}$, where the latter
is some crossed product  of the group $G/N$ over the ring $RN$.

If $R$ is any ring and $C$ denotes an infinite cyclic group then any
crossed product $RC\cong R[x,x^{-1};\alpha]$ for a suitable ring
automorphism $\alpha \colon R\to R$ given by conjugation by $x$.

We refer the reader to  \cite{Passman1} for further details on
crossed products. If $k$ is a commutative field and $R$ is a
$k$-algebra, then the construction of $RG$ is a particular case of a
Hopf algebra crossed product, see \cite[Chapter 7]{Montgomery}

\medskip

We say that a group $G$ is an \emph{orderable group} if there exists
a total order $<$ on $G$ which is compatible with the product
defined on $G$, that is,  $x<y$ implies that $zx<zy$ and $xz<yz$ for
all $x,y,z\in G$. In this event $(G,<)$ is an \emph{ordered group}.

\medskip

Given a ring $R$, an ordered group $(G,<)$ and a crossed product
group ring $RG$, the \emph{Malcev-Neumann series ring} $R((G,<))$
consists of the formal sums
$$f=\sum_{x\in G}a_x\bar{x},$$ such that $\supp f=\{x\in G\mid a_x\neq 0\}$ is a
well-ordered subset of $G$, the sum is defined componentwise and the
product is defined as in \eqref{eq:product}.

It was proved independently by A.I. Malcev \cite{Malcev} and B.H.
Neumann \cite{Neumann} that if $R$ is a field then $R((G,<))$ is
also a field. Let $f=\sum_{x\in G}a_x\bar{x}$ be a
nonzero series in $R((G,<))$. Set $x_0=\min\{x\in G\mid x\in\supp f\}$
and $g=a_{x_0}\bar{x}_0-f$. Observe that $\supp
g(a_{x_0}\bar{x}_0)^{-1}\subseteq\{x\in G\mid x>1\}$. As in
\cite[Corollary~14.23]{Lam1}, it can bee seen that $\sum_{m\geq
0}(g(a_{x_0}\bar{x}_0)^{-1})^m$ is a well-defined element in
$R((G,<))$, that is, for each $x\in G$ the set $L_x=\{m\geq 0\mid
x\in \supp (g(a_{x_0}\bar{x}_0)^{-1})^m\}$ is finite. Then
$$f^{-1}=(a_{x_0}\bar{x}_0)^{-1}\sum_{m\geq 0}\big(g(a_{x_0}\bar{x}_0)^{-1}\big)^m$$

\subsection{Universal fields, matrix localization and the free field} \label{localitzacio}
See \cite[Chapter 4]{Cohnskew} for the missing details. Let $R$ be a ring. An
\emph{epic $R$-field} is a morphism of rings $\iota\colon
R\rightarrow E$ with $E$ a field which is rationally generated by
the image of $\iota$. If $\iota$ is injective, it is called a
\emph{field of fractions} of $R$. It is known that epic $R$-fields
(objects) together with specializations (morphisms) form a category.
If there exists an initial object in this category it is called a
\emph{universal field}. If it exists, it is unique up to
isomorphism.

Observe that an endomorphism $f\colon F\to F$  in the category  of
epic $R$-fields must be an automorphism of $R$-rings. In particular,
epic $R$-fields are isomorphic if and only if they are isomorphic as
$R$-rings.

Let $R$ be a ring and let $E$ be an epic $R$-field with morphism
$\varphi \colon R\to E$. It was proved by Cohn that  the set
$\mathcal{P}_E$ of all square matrices with entries in $R$ and such
that its image via $\varphi$ is not invertible  in $E$ form a
\emph{prime matrix ideal of $R$} and the localization of $R$ at the
set of all square matrices with entries in $R$ such that its image
via $\varphi$ is invertible is a local ring, denoted by
$R_{\mathcal{P}_E}$, such that the canonical map
$R_{\mathcal{P}_E}\to E$ induces an isomorphism between  the residue
field of $R_{\mathcal{P}_E}$ and $E$. Let us call $\mathcal{P}_E$
the associated prime matrix ideal to the epic $R$-field $E$.

This correspondence between   epic $R$-fields and prime matrix
ideals of a ring $R$ is in fact bijective. If  $\mathcal{P}$ is a
prime matrix ideal of $R$ then  $R_\mathcal{P}$ is a local ring, its
residue field $E$ is an epic $R$-field and
$\mathcal{P}_E=\mathcal{P}$.

\begin{theo} Let $R$ be a ring, and let $F_1$ and $F_2$ be epic $R$-fields with associated prime matrix ideals $\mathcal{P}_1$ and $\mathcal{P}_2$, respectively. Then the following statements are equivalent:
\begin{enumerate}
\item[(i)] There exists a specialization $F_1\to F_2$.
\item[(ii)]  $\mathcal{P}_1\subseteq \mathcal{P}_2$.
\item[(iii)]  The canonical localization homomorphism
$R\to R_{\mathcal{P}_1}$ factors through the canonical
localization homomorphism $R\to R_{\mathcal{P}_2}$.
\end{enumerate}
In particular, if $\mathcal{P}_2$ is a minimal prime matrix ideal,
then $\mathcal{P}_1= \mathcal{P}_2$ and $F_1$ is isomorphic to
$F_2$. \end{theo}

Note also that the third statement in the theorem above implies that
if $F_2$ is given by universal localization of $R$ (at a prime
matrix ideal of $R$), then $F_2$ is isomorphic  to $F_1$. Therefore
one can deduce that   the prime matrix ideal associated to $F_2$ is
a minimal prime matrix ideal.

All prime matrix ideals contain the set of non-full matrices. The set $\mathcal{P}$ of non-full matrices is a prime matrix ideal, hence the least prime matrix ideal, if and only if $R$ is a Sylvester domain and, in this case,  $R_\mathcal{P}$ is a field and hence, it is a universal field of fractions.  A free algebra (or more generally a semifir) is a Sylvester domain. The universal field of fractions of a free algebra is usually called a
\emph{free field}.

Let $G$ be a free group on a nonempty set $X$, $k$ a field and $kG$
a crossed product. Lewin proved that the universal field of
fractions of $kG$ (and of $\freealgebra kX$) is the field of
fractions of $kG$ inside $k((G,<))$ for any total order $<$ on $G$
such that $(G,<)$ is an ordered group, see \cite{Lewin} and the
remark in \cite[Section~2]{LewinLewin}. An easier proof of this fact
was given by C. Reutenauer~\cite{Reutenauerseries} (or see
\cite{Tesis}). Observe that if $N$ is a subgroup of $G$ (or
$Y\subseteq X$), then the universal field of fractions of $kN$
(respectively $\freealgebra kY$) is the field of fractions of $kN$
($\freealgebra kY$) inside $k((G,<))$.

\section{Inversion height}

Suppose that  $\iota \colon R\hookrightarrow E$ is an embedding of a
domain $R$ into a field $E$.

Set $E_{\iota}(-1)=\emptyset$, $E_{\iota}(0)=R$, and we define
inductively for $n\geq0$:
 $$E_{\iota}(n+1)=\begin{array}{c}\textrm{\small subring of } E
\\ \textrm{\small generated by}\end{array} \left\{r,
s^{-1}\mid r,s\in E_\iota(n),\ s\neq0\right\}.$$ Then
$E_\iota=\operatornamewithlimits{\bigcup}\limits_{n=0}^\infty
E_\iota(n)$ is the \emph{field of fractions of $R$ inside $E$}. That
is, $E_\iota$ is the field rationally generated by $R$ inside $E$
or, equivalently, the intersection of all subfields of $E$ that
contain $R$.

 We define $\h_\iota (R)$, the \emph{inversion height of $R$}
(inside $E$), as $\infty$ if there is no $n\in\naturals$ such that
$E_\iota(n)$ is a field. Otherwise,
$$\h_\iota (R)=\min\{n\mid E_\iota(n) \textrm{ is a field}\}.$$
Notice that if $\h_\iota(R)=n$, then $E_\iota(m)=E_\iota(n)$ for all
$m\geq n$.

 Given an integer $n\geq 0$, we say that an element $f\in E_\iota$
has  \emph{inversion height} $n$ if $f\in E_\iota(n)\setminus
E_{\iota}(n-1),$ and we write $\h _\iota (f)=n$. In other words,
$\h_\iota(f)$ says how many nested inversions are needed to express
an element of $E_\iota$ from elements of $R$, and $\h _\iota(R)$ is
the supremum of all $\h_\iota(f)$ with $f\in E_\iota$.

\medskip

We now give some easy remarks that will be used throughout.

\begin{rems}\label{rems:usedalot}
Let $\iota \colon R\hookrightarrow E$ be an embedding of a domain
$R$ in a field $E$.

\begin{enumerate}[(a)]
\item If $\kappa\colon E\hookrightarrow L$ is an embedding  in a field $L$, then $\kappa \iota$ is an embedding  such that $E_\iota(n)=L_{\kappa\iota}(n)$ for all
$n\geq -1$. Therefore $E_\iota=L_{\kappa\iota}$,
$\h_\iota(R)=\h_{\kappa\iota}(R)$, and
$\h_\iota(f)=\h_{\kappa\iota}(f)$ for all $f\in L_{\kappa\iota}$.

\item On the other hand, if $S$ is a subring of $R$ and we consider the embedding
$\varepsilon=\iota_{\mid S}\colon S\hookrightarrow E$, then
$E_\varepsilon(n)\subseteq E_\iota(n)$, and thus $\h_\iota(f)\leq
\h_\varepsilon(f)$ for all $f\in E_\varepsilon$.
\end{enumerate}
\end{rems}

One of the  problems when dealing with inversion height is the fact
that  we cannot be more accurate in Remarks~\ref{rems:usedalot}(b).
That is, we may know $\h_\varepsilon(f)$ for some $f$ or even
$h_\varepsilon(S)$, but usually it is not useful if we want to
compute $\h_\iota(f)$ or $\h_\iota(R)$. Our key  results on
inversion height (Propositions~\ref{prop:heightskewpolynomial} and
\ref{prop:heightskewpolynomial2}) state that
$\h_\varepsilon(f)=\h_\iota(f)$ in certain important cases.

\begin{lem}\label{lem:finiteset} Let $k$ be a commutative field,
and let $R$ be a $k$-algebra with a fixed embedding $\iota\colon
R\hookrightarrow E$ into a field $E$. If $f\in E_\iota$ satisfies
that $h_\iota (f)\le m$, then there exists a finitely generated
$k$-subalgebra $S$ of $R$ such that $f\in E_{\varepsilon}$ and
$h_{\varepsilon}(f)\le m$ where $\varepsilon=\iota_{\mid S} \colon
S\to E$.
\end{lem}

\begin{proof} The proof is by induction on $m$.
For $m=0$ the claim is clear. Suppose that the claim is true for
$m-1\geq 0.$ Since $f\in E_\iota(m),$  $f=\sum\limits_{j=1}^r
f_{1j}\cdots f_{l_jj}$ where, for each $i,j$, either $f_{ij}\in
E_{\iota}(m-1)$ or $f_{ij}$ is the inverse of some nonzero element
in $E_{\iota}(m-1).$ The induction hypothesis implies that there
exist $S_{1j},\dotsc,S_{l_jj}$ finitely generated $k$-subalgebras of
$R$ such that $f_{ij}\in E_{\varepsilon _{ij}}$, where $\varepsilon
_{ij}=\iota_{\mid S_{ij}}\colon S_{ij}\to E$, and $h_{\varepsilon
_{ij}}(f_{ij})\leq m.$ Let $S$ be the smallest subalgebra of $R$
containing $S_{ij}$ for all $i$, $j$, and let
$\varepsilon=\iota_{\mid S} \colon S\to E$. Then $f\in
E_{\varepsilon}$, and  $h_{\varepsilon} (f)\leq m$ because
$E_{\varepsilon_{ij}}(m)\subseteq E_{\varepsilon}(m).$ This proves
the result.
\end{proof}

\begin{lem}\label{lem:automorphism} Let  $S$ be a domain with
a fixed embedding $\varepsilon \colon S\hookrightarrow F$ into a
field $F$. Let $\alpha \colon F\to F$ be a morphism of rings and
$\delta\colon F\rightarrow F$ be an $\alpha$-derivation.
\begin{enumerate}[\rm(i)]
\item If
 $\alpha(S)\subseteq S$ and $\delta(S)\subseteq S$, then
$$\alpha (F_\varepsilon(n))\subseteq F_\varepsilon(n)\quad
\textrm{and} \quad \delta(F_\varepsilon(n))\subseteq
F_\varepsilon(n)$$
 for all
$n\ge 0$. Hence, $F_\varepsilon(n)((y;\alpha,\delta))\hookrightarrow
F_\varepsilon ((y;\alpha,\delta))$ and
$F_\varepsilon(n)((x;\alpha))\hookrightarrow
F_\varepsilon((x;\alpha))$.

\item If $\alpha (S)=S$, then $\alpha$ induces an automorphism of
$F_\varepsilon(n)$  for each $n\ge 0$, and thus it induces an
automorphism on $F_\varepsilon$.
\end{enumerate}
\end{lem}

\begin{proof}
(i) The hypothesis ensures that $\alpha\left(
F_\varepsilon(0)\right)\subseteq F_\varepsilon(0)$ and $\delta\left(
F_\varepsilon(0)\right)\subseteq F_\varepsilon(0)$. Since for each
$f\in F\setminus\{0\}$, $\alpha(f^{-1})=\alpha(f)^{-1}$ and
$\delta(f^{-1})=-\alpha(f)^{-1}\delta(f)f^{-1}$ cf. Lemma
\ref{lem:uniqueextensionderivation}, using the definition of
$F_\varepsilon(n)$, it is easy to prove the first claim inductively.

The second claim follows from the first and the commutativity of the
following diagram
$$\xymatrix{ F_\varepsilon(n)[[y;\alpha,\delta]]\ar@{^{(}->}[r]^\eta \ar@{^{(}->}[d]  &
F_\varepsilon [[y;\alpha,\delta]]\ar@{^{(}->}[d]\\
F_\varepsilon(n)((y;\alpha,\delta)) \ar@{^{(}->}[r]^\nu &
F_\varepsilon((y;\alpha,\delta))}$$ where the vertical arrows are
given by the right Ore localization at the powers of $y$, $\eta$ is
induced from $F_\varepsilon(n)\hookrightarrow F_\varepsilon$, and
$\nu$ is given by the universal property of Ore localization.
Similarly for $F_\varepsilon(n)((x;\alpha))$.

(ii) Assume that $\alpha \colon S\to S$ is an automorphism. We
prove, by induction on $n$, that  $\alpha \colon F_\varepsilon(n)\to
F_\varepsilon(n)$ is an isomorphism for each $n\ge 0$. Our
hypothesis ensures the case $n=0$. Assume that $n>0$ and  $\alpha
\colon F_{\varepsilon}(n-1)\to F_{\varepsilon}(n-1)$ is onto, hence
an automorphism. As for any $r\in F_{\varepsilon}(n-1)\setminus
\{0\}$, $\alpha (r^{-1})=\alpha(r)^{-1}\in F_{\varepsilon}(n)$ and
$F_{\varepsilon}(n-1)= \alpha \left(F_{\varepsilon}(n-1)\right)$, we
deduce that all the ring generators of $F_{\varepsilon}(n)$ are in
$\alpha (F_{\varepsilon}(n))$, which implies that $\alpha \colon
F_\varepsilon(n)\to F_\varepsilon(n)$ is onto.
\end{proof}

\begin{prop}\label{prop:heightskewpolynomial}
Let $S$ be a domain, let $\alpha\colon S\rightarrow S$ be an
injective ring endomorphism, and let $\delta\colon S\rightarrow S$
be  an $\alpha$-derivation. Set  $R=S[x;\alpha,\delta]$. Suppose
that $\varepsilon\colon S\hookrightarrow F$ is a field of fractions
of $S$, that $\alpha$ and $\delta$ extend to $F$ and that
\begin{equation}\label{eq:inversionheightskewpolynomial}
\alpha\big(F_\varepsilon(n)\setminus
F_\varepsilon(n-1)\big)\subseteq F_\varepsilon(n)\setminus
F_\varepsilon(n-1),
\end{equation}
for each integer $n\geq0$. Let  $E=F(x;\alpha,\delta)$, and let
$\iota\colon R\hookrightarrow E$ be the natural embedding of $R$ in
$E$. Consider the field of skew Laurent series
$F((y;\alpha,\delta))$. Then

\begin{enumerate}[\rm(i)]
\item For each $n\geq 0$,
$E_{\iota}(n)\subseteq F_\varepsilon(n)((y;\alpha,\delta))$.

\item Let $f\in F$. If $\h_\varepsilon(f)=n$, then
$\h_\iota(f)=n$.
\item $\h_{\iota}(R)\geq \h_{\varepsilon}(S)$.
\end{enumerate}
\end{prop}

\begin{proof}
To simplify the notation, let
$\mathcal{L}_n=F_\varepsilon(n)((y;\alpha,\delta))$ for each $n\geq
0$. By Lemma~\ref{lem:automorphism}(i), we may consider
$\mathcal{L}_n$ as a subring of $F((y;\alpha,\delta))$.

(i) We proceed by induction on $n$. For $n=0$, observe that
$E_\iota(0)=S[x;\alpha,\delta]$. Given $f=a_0+a_1x+\dotsb+a_nx^n\in
S[x;\alpha,\delta]$, it can be expressed as
$(a_0y^n+a_1y^{n-1}+\dotsb+a_n)y^{-n}$. Now $a_0y^n,\dotsc,a_n\in
S[[y;\alpha,\delta]]$ by \eqref{eq:commutationruleseries}. Thus
$E_\iota(0)=S[x;\alpha,\delta]\subseteq
S((y;\alpha,\delta))=\mathcal{L}_0$. Suppose that the result holds
for $n\geq 0$. Let $f\in E_\iota(n)\setminus\{0\}$. Express $f$ as
an element in $\mathcal{L}_n$,
$f=(\sum\limits_{m\geq0}y^ma_m)y^{-r}$. Suppose that $m_0$ is the
first natural such that $a_{m_0}\neq 0$. Then $f$ can be written as
$y^{m_0}\big(1-\sum\limits_{m\geq 1}y^mb_m\big)a_{m_0}y^{-r}$ where
$b_m=-a_{m+m_0}a_{m_0}^{-1}$. Hence
\begin{equation}\label{eq:inverse}
f^{-1}=y^ra_{m_0}^{-1}\left(\sum_{s\geq0}\left(\sum_{m\geq
1}y^mb_m\right)^{\! \!s}\, \right)y^{-m_0}.
\end{equation}
Observe that for each $s\geq 0$, the terms from $\big(\sum_{m\geq 0}
y^mb_m\big)^t$ with $t>s$ do not contribute to the coefficient of
$y^s$, and the coefficient of $y^s$ belongs to $F_\varepsilon(n+1)$
by Lemma~\ref{lem:automorphism}(i). Hence
$\sum\limits_{s\geq0}\big(\sum\limits_{m\geq1}y^mb_m\big)^s\in
F_\varepsilon(n+1)[[y;\alpha,\delta]]$. Now it is easy to prove that
$f^{-1}\in\mathcal{L}_{n+1}$.

Since $E_\iota(n)\subseteq E_{\iota}(n+1)$, we have shown that the
generators of $E_\iota(n+1)$ are contained in the ring
$\mathcal{L}_{n+1}$. Therefore
$E_\iota(n+1)\subseteq\mathcal{L}_{n+1}$, as desired.

(ii) If $S$ is a field, the result is clear. So suppose that $S$ is
not a field and let $f\in F$ with $f\in F_\varepsilon(n+1)\setminus
F_\varepsilon (n)$ for some $n\geq 0$. If $f\in\mathcal{L}_n$, i.e.
$f=\big(\sum_{m\geq 0}y^ma_m\big)y^{-r}$ with $a_m\in
F_\varepsilon(n)$, then $fy^{r}=\sum_{m\geq 0}y^ma_m$. On the one
hand $fy^r$ is a series of the form $y^r\alpha^r(f)+\sum_{m\geq
1}y^{r+m}b_m$. Since $a_m\in F_\varepsilon(n)$ for
all $m\geq 0$ this is a contradiction because $a_r=\alpha^r(f)\in
F_\varepsilon(n+1)\setminus F_\varepsilon(n)$ by the hypothesis
\eqref{eq:inversionheightskewpolynomial}.

(iii) follows from (ii).
\end{proof}

Note that if $\alpha$ is an automorphism, then
\eqref{eq:inversionheightskewpolynomial} in
Proposition~\ref{prop:heightskewpolynomial} holds.

\begin{prop}\label{prop:heightskewpolynomial2}
Let $S$ be a domain, $\alpha\colon S\rightarrow S$ be an
automorphism and $R=S[x,x^{-1};\alpha]$. Suppose that
$\varepsilon\colon S\hookrightarrow F$ is a field of fractions of
$S$ and that $\alpha$
 extends to $F$.
Let  $E=F(x;\alpha)$  and $\iota\colon R\hookrightarrow E$ be the
natural embedding of $R$ in $E$. Consider the field of skew Laurent
series $F((x;\alpha))$. Then

\begin{enumerate}[\rm(i)]
\item For each $n\geq 0$,
$E_{\iota}(n)\subseteq F_\varepsilon(n)((x;\alpha))$.

\item Let $f\in F$. If $\h_\varepsilon(f)=n$, then
$\h_\iota(f)=n$.
\item $\h_{\iota}(R)\geq \h_{\varepsilon}(S)$.
\end{enumerate}
\end{prop}

\begin{proof}
Consider $\mathcal{L}_n=F_\varepsilon(n)((x;\alpha))$ as a subring
of $F((x;\alpha))$. Then proceed as in the proof of
Proposition~\ref{prop:heightskewpolynomial}
\end{proof}

\section{Two solutions} \label{sec:elementary}

We shall use the following notation. Let $A$ be an $n\times n$
matrix with entries over a ring. Let $i,j,p,q\in\{1,\dotsc,n\}$. By
$A^{ij}$  we denote the matrix obtained from $A$ by deleting the
$i$-th row and the $j$-th column. By
$r_p^j$\label{notation:rowwithoutentry} we mean the row vector
obtained from the $p$-th row of $A$ deleting the $j$-th entry. And
by $s_q^i$\label{notation:columnwithoutentry} we denote the column
vector obtained from the $q$-th column of $A$ by deleting the $i$-th
entry.

Let $k$ be a commutative field and $X$ a set. Let $A=(x_{ij})$ be an
$n\times n$ matrix with entries over the free $k$-algebra
$\freealgebra kX$. We say that $A$ is a \emph{generic matrix} (over
$\freealgebra kX$) if the $x_{ij}$'s are distinct variables in $X$.
If $\iota\colon \freealgebra kX\hookrightarrow E$ is the universal
field of fractions of $\freealgebra kX$, then such
a generic matrix is invertible over $E$. Moreover the $(j,i)$-th
entry of $A^{-1}\in M_n(E)$  is the inverse of
$$|A|_{ij}=x_{ij}-r_i^j(A^{ij})^{-1}s_j^i.$$  The element $|A|_{ij}$
is known as the $(i,j)$-th \emph{quasideterminant} of $A$
\cite{Gelfandretakhwilson}.

\begin{theo}\label{theo:reutenauerinversionheight}
\emph{(C.~Reutenauer \cite[Theorem~2.1]{Reutenauerinversionheight})}
Let $k$ be a commutative field and let $X$ be a finite set of
cardinality at least $n^2$, where $1\leq n<\infty$. Let $\iota\colon
\freealgebra kX\hookrightarrow E$ be the embedding of the free
algebra $\freealgebra kX$ in its universal field of fractions $E$.
Let $A$ be an $n\times n$ generic matrix. If $f$ is an  entry of
$A^{-1}\in M_n(E)$, then $\h_\iota(f)=n$.\qed
 \end{theo}

To adapt this result to our purposes, we note the following
Corollary.

\begin{coro}\label{coro:matrixheight} Let $k$ be a commutative field,
let $Z$ be an infinite set and let $N$ be the free group on $Z$. Let
$\varepsilon'\colon  k[N]\hookrightarrow F$ be the universal field
of fractions of the group algebra $k[N]$, and
$\varepsilon=\varepsilon'_{\mid \freealgebra kZ}\colon \freealgebra
kZ\hookrightarrow F$. Then
$\h_{\varepsilon'}(k[N])=\h_\varepsilon(\freealgebra kZ)=\infty$.
Indeed, if $A_n$ is an $n\times n$ generic matrix and $f$ is an
entry of $A_n^{-1}\in M_n(F)$, then $\h_\varepsilon(f)=n$ and
$\h_{\varepsilon'}(f)=n-1$.
\end{coro}

\begin{proof}
First of all notice that since $\freealgebra kZ\subseteq k[N]$,
\begin{equation}\label{eq:inversionheightgeneric}
E_{\varepsilon}(m)\subseteq E_{\varepsilon'}(m)\subseteq
E_\varepsilon(m+1)\subseteq E_{\varepsilon'}(m+1) \end{equation} for
each integer $m\geq 0$. Thus if $\h_{\varepsilon}(\freealgebra
kZ)=\infty$, then $\h_{\varepsilon'}(k[N])=\infty$.

Let $A_n$ be an $n\times n$ generic matrix.  Recall that if $Y$ is a
subset of $Z$ and $\eta=\varepsilon_{\mid \freealgebra
kY}\colon\freealgebra kY\hookrightarrow F$, then $F_\eta$ is the
universal field of fractions of $\freealgebra kY$. Thus if $Y$ is
any finite subset of $Z$ that contains the entries of $A_n$ and $f$
is an entry of $A_n^{-1}$, then $\h_\eta(f)=n$ by
Theorem~\ref{theo:reutenauerinversionheight}. Now
Lemma~\ref{lem:finiteset} implies that $\h_\varepsilon(f)=n$, and
\eqref{eq:inversionheightgeneric} that $\h_{\varepsilon'}(f)\geq
n-1$.

Since $Z$ is an infinite set, there exist $n\times n$ generic
matrices $A_n$ for each natural $n\geq 1$ and therefore
$\h_\iota(\freealgebra kZ)$ is not finite by the foregoing.

We prove that $\h_{\varepsilon'}(f)\leq n-1$ by induction on $n\geq
1$. If $n=1$, the result follows because $f\in Z$ and therefore
$f^{-1}\in N$. Suppose the claim holds for $n\geq 1$. Consider an
$(n+1)\times (n+1)$ generic matrix $A_{n+1}=(x_{ij})$. Then $f$ is
the $(j,i)$-th entry of $A_{n+1}^{-1}$. Thus
$f=\Big(x_{ij}-r_i^j(A_{n+1}^{ij})^{-1}s_j^i \Big)^{-1}$ for some
$i,j$. Since $A_{n+1}^{ij}$ is an $n\times n$ generic matrix, the
induction hypothesis implies that if $g$ is any entry of
$(A_{n+1}^{ij})^{-1}$ then $\h_{\varepsilon'}(g)\leq n-1$. Therefore
$\h_{\varepsilon'}(f)\leq n$.
\end{proof}

\subsection{First solution}

If $x,\ y$ are two elements of a ring, we denote by $[x,y]$ the
element $[x,y]=xy-yx$.

We are interested in extending derivations to certain localizations
of $R$. We recall the following easy and well known formula which
implies that such extensions, if they exist, are unique.

\begin{lem} \label{lem:uniqueextensionderivation}
Let $R$ be a ring, and let $\delta\colon R\to R$ be a derivation. If
$r\in R$ is invertible, then $\delta
(r^{-1})=-r^{-1}\delta(r)r^{-1}$. Hence, if $R\rightarrow D$ is a
ring extension such that $D$ is a field of fractions of $R$ and
$\delta$, $\delta '\in \mathrm{Der}\, (D)$ are such that $\delta
(r)=\delta '(r)$, for any $r\in R$, then $\delta =\delta '$.
\end{lem}

In the next lemma, we show that derivations can be extended to
matrix localizations provided the set $\Phi$ we localize at is
\emph{upper multiplicative}, that is, $1\in \Phi$ and whenever
$A,B\in\Phi$, then $\left(\begin{smallmatrix} A & C \\ 0 & B
\end{smallmatrix}\right)\in \Phi$ for any matrix $C$ of
appropriate size. The result, at least for fields of fractions of
Sylvester domains, is well known and the proof for the general case
follows the same pattern. However we include the proof for
completeness' sake.

Recall that if $R$ is a ring, $\delta \colon R\to R$ is a derivation if and only if the map $R\to M_2(R)$ given by $r\mapsto \begin{pmatrix}r&\delta (r)\\ 0&r \end{pmatrix}$, for any $r\in R$, is a ring homomorphism.

For the proof of the next result it is useful to keep in mind the
following explicit description of an isomorphism between $M_{2n}(S)$
and $M_n(M_2(S))$ for any natural number $n$ and any given ring $S$.
The elements of $M_n(M_2(S))$ are matrices of the form
$$A=\left(\begin{array}{ccc} A_{11} &
\cdots & A_{1n} \\ \vdots & \ddots & \vdots \\
A_{n1} & \cdots & A_{nn}
\end{array}\right)$$
where $A_{ij}=\left(\begin{smallmatrix} a_{ij} & b_{ij} \\ c_{ij} &
d_{ij}
\end{smallmatrix}\right)\in M_2(S)$ for each $i,j\in\{1,\dotsc,n\}$.
The map $\rho_n\colon M_n(M_2(S))\rightarrow M_{2n}(S)$ defined by
$$\left(\begin{array}{ccc} A_{11} & \cdots & A_{1n} \\
\vdots & \ddots & \vdots \\
A_{n1} & \cdots & A_{nn} \end{array}\right)\mapsto
\left(\begin{array}{cccccc}
a_{11} & \cdots & a_{1n} & b_{11} &
\dotsc & b_{1n} \\
\vdots & \ddots & \vdots & \vdots & \ddots & \vdots \\
a_{n1} & \cdots & a_{nn} & b_{n1} &
\dotsc & b_{nn} \\
c_{11} & \cdots & c_{1n} & d_{11} &
\dotsc & d_{1n} \\
\vdots & \ddots & \vdots & \vdots & \ddots & \vdots \\
c_{n1} & \cdots & c_{nn} & d_{n1} & \dotsc & d_{nn}
\end{array}\right)$$
is an isomorphism of rings.

\begin{lem}\label{lem:extensionofderivations}
Let $R$ be a ring, $\Phi$ an upper multiplicative set of square
matrices over $R$, and let $R\rightarrow R_\Phi$, $a\mapsto
\hat{a}$, be the matrix localization of $R$ at $\Phi$. Then any
derivation $\delta\colon R\rightarrow R$, $a\mapsto a^\delta$,
extends to a unique derivation of $R_\Phi$.

In particular, if $R\hookrightarrow D$ is the universal field of
fractions of a Sylvester domain $R$, then any derivation in $R$ can
be uniquely extended to $D$.
\end{lem}

\begin{proof} We suppose that $R\rightarrow R_\Phi$ is given by $a\mapsto
\hat{a}$. For each matrix $A=(a_{ij})\in M_n(R)$, denote by
$\hat{A}=(\hat{a}_{ij})\in M_n(R_\Phi)$ and by $A^\delta$ the matrix
$(a_{ij}^\delta)\in M_n(R)$.

For each natural $n$, consider the map $\psi_n\colon
M_n(R)\rightarrow M_{2n}(R_\Phi)$ given by $A\mapsto
\left(\begin{smallmatrix} \hat{A} & \widehat{A^\delta} \\
0 & \hat{A}
\end{smallmatrix}\right)$. Since $\delta$ is a derivation, $\psi_n$
is a morphism of rings. For each $n\times n$ matrix $A\in \Phi$, the
matrix $\psi_n A$ is invertible in $M_{2n}(R_\Phi)$. Indeed, since
$\hat{A}$ is invertible in $R_\Phi$ by definition, the matrix
\begin{equation}\label{eq:inversematrix}
\left(\begin{array}{cc} \hat{A}^{-1} & -\hat{A}^{-1}
\widehat{A^\delta}\hat{A}^{-1} \\ 0 & \hat{A}^{-1}\end{array}\right)
\end{equation}
is the inverse of $\psi_nA$. Thus the image of any $n\times n$ matrix
in $\Phi$ by the morphism  $\rho_n^{-1}\psi_n\colon
M_n(R)\rightarrow M_n(M_2(R_\Phi))$ is invertible. Note that if
$A=(a_{ij})$, then
\begin{equation}\label{eq:image2nmatrices}
\rho_n^{-1}\psi_n A=\left(\begin{array}{ccc}A_{11} & \cdots & A_{1n} \\
\vdots & \ddots & \vdots \\
A_{n1} & \cdots & A_{nn}
\end{array}\right)
\end{equation}
where $A_{ij}=\left(\begin{smallmatrix} \hat{a}_{ij} &
\widehat{a_{ij}^\delta} \\ 0 & \hat{a}_{ij}
\end{smallmatrix}\right)$. Hence, the morphism $R\rightarrow
M_2(R_\Phi)$, $a\mapsto \left(\begin{smallmatrix} \hat{a} &
\widehat{a^\delta} \\ 0 & \hat{a}
\end{smallmatrix}\right)$ is $\Phi$-inverting, and there exists a
unique morphism $R_\Phi\rightarrow M_2(R_\Phi)$ making the diagram
$$\xymatrix{ R\ar[rr]\ar[dr]  & & M_2(R_{\Phi}) \\
 & R_\Phi\ar[ur]\\
}$$ commutative. Note that $R_\Phi$ is the $\Phi$-rational closure
of $R$ in $R_\Phi$
\cite[Theorem~7.12]{Cohnfreeeidealringslocalization}. Thus, for any
element $x\in R_\Phi$, there is some $A\in\Phi$ such that $x$ is an
entry of the inverse matrix of $\hat{A}$. Looking at
\eqref{eq:inversematrix} and \eqref{eq:image2nmatrices}, we see that
the image of $x\in M_2(R_\Phi)$ is of the form
$\left(\begin{smallmatrix} x & x^\Delta \\ 0 & x
\end{smallmatrix}\right)$ for some $x^\Delta\in R_\Phi$. Hence
$\Delta\colon R_\Phi\rightarrow R_\Phi$, $x\mapsto x^\Delta$, is a
derivation extending $\delta$, as desired.

For the last part, it is known that if $R$ is Sylvester domain, then
its universal field of fractions is of the form $R_\Phi$ where
$\Phi$ is the set of all full matrices over $R$. Note that $\Phi$ is
upper multiplicative because it is the set of matrices over $R$ that
become invertible in its universal field of fractions.
\end{proof}

Next result is based on the ideas of  \cite{Cohnembeddingtheorem},
where a particular kind of embedding of a free algebra of infinite
countable rank into free algebra of rank two is given.

\begin{theo} \label{theo:firstsolution}
Let $k$ be a commutative field and $\freealgebra
k{x,y_1,\dotsc,y_n}$ be the free algebra with $n\geq 1$. Let
$\iota\colon \freealgebra k{x,y_1,\dotsc,y_n}\hookrightarrow E$ be
the universal field of fractions of $\freealgebra
k{x,y_1,\dotsc,y_n}$. Then $\h_\iota (\freealgebra
k{x,y_1,\dotsc,y_n})=\infty$. Moreover, if
$$A_m=\left(\begin{array}{cccc}w_0 & w_m & \cdots
& w_{m^2-m} \\ w_1 & w_{m+1} & \cdots & w_{m^2-m+1}\\
\vdots & \vdots & \ddots &\vdots \\
w_{m-1} & w_{2m-1} & \cdots & w_{m^2-1}
\end{array}\right),$$ where
$$w_0=y_1, \quad w_i=[x,\dotsc[x,[x,y_1]]\dotsb] \ \textrm{ with $i$ factors }
x,$$ and $f$ is an entry of $A_m^{-1}\in M_m(E)$, then $\h_\iota
(f)=m$.
\end{theo}

\begin{proof}
Set $Z=\{z_0,z_1,\dotsc,z_m,\dotsc\}$, $S=\freealgebra kZ$,
$R=\freealgebra k{x,y_1,\dotsc,y_n}$ and $\varepsilon \colon
S\hookrightarrow F$ the universal field of fractions of $S$.

Proceeding as in \cite[Lemma~2.1]{Cohnembeddingtheorem} or  using
Lemma \ref{lem:extensionofderivations}, it can be seen that there
exists a derivation $\delta\colon S\rightarrow S$ such that $\delta
(z_i)=z_{i+n}$, for each $i\in\naturals$, and that it can be
extended to a unique derivation of $F$.

Express each integer $i\geq0$ (uniquely) as $i=rn+j$ with $0\leq
j\leq n-1$. As in \cite[Theorem~2.2]{Cohnembeddingtheorem}, one can
prove that there is an embedding $\beta_0\colon S\rightarrow R$
defined by
$$\beta_0 (z_i)=\left\{\begin{array}{ll}  y_{i+1} & \textrm{ for }
0\leq i\leq n-1\\ {}  [x,\dotsc[x,[x,y_{j+1}]]\dotsb] \textrm{ with
$r$ factors }x & \textrm{ for } n-1< i.
\end{array}\right.$$
which is honest (and 1-inert). Thus $\beta_0$ can be extended to a
morphism of rings $\beta_0\colon F\hookrightarrow E$. Again as in
\cite{Cohnembeddingtheorem}, identifying $S$ and $F$ with their
images via $\beta_0$, we get that $R=S[x;\delta]$ and that
$E=F(x;\delta)$. Since $\h_\varepsilon(S)=\infty$ by
Corollary~\ref{coro:matrixheight}, also $\h_\iota(R)=\infty$ by
Proposition~\ref{prop:heightskewpolynomial}(iii).

Now let  $f$ be an entry of the inverse of $A_m$. Note that the
matrix $A_m$ is (the image of) a generic matrix over $S$. Thus
Corollary~\ref{coro:matrixheight} says that $\h_\varepsilon(f)=m$.
Therefore $\h_\iota(f)=m$ by
Proposition~\ref{prop:heightskewpolynomial}(ii).
\end{proof}

\subsection{Second solution}

\begin{theo}\label{theo:solutiontoNeumann}
Let $k$ be a commutative field, $X=\{x,y_1,\dotsc,y_n\}$ be  a
finite set with $n\geq 1$, and $H$ be the free group on $X$. Let
$\iota'\colon k[H]\hookrightarrow E$ be the universal field of
fractions of $k[H]$, and $\iota=\iota'_{\mid \freealgebra
kX}\colon\freealgebra kX\hookrightarrow E$. Then
$\h_{\iota}(\freealgebra kX)=\h_{\iota'}(k[H])=\infty$. Moreover, if
$$A_m=\left(\begin{array}{cccc} z_0 & z_m &
\cdots & z_{m^2-m} \\ z_1 & z_{m+1} & \cdots &
z_{m^2-m+1}\\ \vdots & \vdots & \cdots & \vdots \\
z_{m-1} & z_{2m-1} & \cdots & z_{m^2-1}
\end{array}\right)\ \textrm{ where }   z_i=x^iy_1x^{-i},$$
and $f$ is an entry of $A_m^{-1}\in M_m(E)$, then $\h_{\iota}(f)=m$
and $\h_{\iota'}(f)=m-1$.
\end{theo}

\begin{proof}
Fix an order on $H$ such that $(H,<)$ is an ordered group. We
identify $E$ with the field of fractions of $k[H]$ inside
$L=k((H,<))$.

Let $C=\langle c\rangle$ be the infinite cyclic group. Consider the
morphism of groups $\varphi\colon H\rightarrow C$ given by $x\mapsto
c$ and $y_j\mapsto 1$ for $1\leq j\leq n$. Let $N=\ker \varphi.$
Thus $H$ is the extension of $N$ by the infinite cyclic group
generated by $x$.  It is well known that $N$ is a free group with
basis the infinite set $Z=\{x^iy_jx^{-i}\mid 1\leq j\leq n,\
i\in\integers\}$, see for example \cite[Section~36]{Kurosh}.

 Let
$\varepsilon'=\iota'_{\mid k[N]}\colon  k[N]\hookrightarrow F$ and
$\varepsilon=\iota'_{\mid \freealgebra kZ}\colon \freealgebra
kZ\hookrightarrow F$ be the universal field of fractions of $k[N]$
and $\freealgebra kZ$ respectively, where we identify $F$ with
$E_\varepsilon=E_{\varepsilon'}$, the subfield rationally generated
by $k[N]$ inside $E$.

Let $\alpha\colon E\rightarrow E$ be the automorphism of $E$ given
by $f\mapsto xfx^{-1}$ for all $f\in E$. Notice that $\alpha$
restricts to an automorphism of $k[N]$ and also to an automorphism
of $\freealgebra kZ$. Then $\alpha$ can be extended to an
automorphism of $F$ by Lemma~\ref{lem:automorphism}(ii). Notice also
that $k[H]=k[N][x,x^{-1};\alpha]$, cf. \S \ref{sec:1.1}. Let
$\iota_Z=\iota'_{\mid \freealgebra kZ[x,x^{-1};\alpha]}\colon
\freealgebra kZ[x,x^{-1};\alpha]\rightarrow E$.

Observe that $F$ is contained in $k((N,<))\subseteq L$. Since
$n_1x^{r_1}=n_2x^{r_2}$ for $n_1,n_2\in N$ and $r_1,r_2\in\integers$
if and only if $n_1=n_2$ and $r_1=r_2$, the powers of $x$ are
$k((N,<))$-linearly independent. In particular the powers of $x$ are
$F$-linearly independent. Therefore $\Upsilon\colon
F[x,x^{-1};\alpha]\hookrightarrow E$ and, by the universal properties of the Ore localization, $E=F(x;\alpha)$.

Note that the entries of $A_m$ belong to $Z$. Let $f\in F$, be one of the entries of $A_m^{-1}$.
By Corollary~\ref{coro:matrixheight}, $\h_{\varepsilon'}(f)=m-1$.
Now, if we set $S=k[N]$ and $R=k[H]=k[N][x,x^{-1};\alpha]$,
Proposition~\ref{prop:heightskewpolynomial2}(ii) implies that
 $h_{\iota'}(f)=m-1$.

Similarly, by Corollary~\ref{coro:matrixheight},
$\h_{\varepsilon}(f)=m$. Now, if we set $S=\freealgebra kZ$ and
$R=\freealgebra kZ[x,x^{-1};\alpha]$,
Proposition~\ref{prop:heightskewpolynomial2}(ii) implies that
$h_{\iota_Z}(f)=m$.

Since $\freealgebra kX\subseteq\freealgebra
kZ[x,x^{-1};\alpha]\subseteq k[H]$ we obtain that $$E_\iota
(m-1)\subseteq E_{\iota_Z}(m-1)\subseteq E_\iota(m)\subseteq
E_{\iota_Z}(m),$$ $$E_\iota(m-1)\subseteq E_{\iota'}(m-1)\subseteq
E_\iota(m)\subseteq E_{\iota'}(m).$$ The first expression says that
$m\leq\h_{\iota}(f)$, and the second one $\h_{\iota}(f)\leq m$.
Therefore $\h_{\iota}(f)=m$.

Since $m$ is any natural number $\geq 1$, we obtain that there exist
elements $f\in E$ with any prescribed inversion height $m\geq 1$.
Therefore $\h_\iota(\freealgebra kX)=\h_{\iota'}(k[H])=\infty$.
\end{proof}


\section{Other embeddings of infinite inversion height} \label{sec:other}

Let $S$ be a ring, $G$ a group and $SG$ a crossed product
(determined by maps $\sigma$ and $\tau$ as in
section~\ref{subsec:crossedproducts}). Let $\varepsilon\colon
S\hookrightarrow F$ be an epimorphism of rings such that the
automorphism $\sigma(x)\in\Aut(R)$ can be extended to an
automorphism of $F$ for every $x\in G$. It is easy to prove, for
example as in \cite[Lemma~4]{SanchezfreegroupalgebraMNseries}, that
there exists a crossed product $FG$ with an embedding $\kappa\colon
SG\rightarrow FG$ with $\kappa_{\mid S}=\varepsilon$ and
$\kappa(\bar{x})=\bar{x}$.

If $\varepsilon\colon S\hookrightarrow F$ is a field of fractions,
then it is easy to prove that $\varepsilon$,  $\varepsilon_n\colon
S\hookrightarrow F_\varepsilon(n)$ and
$F_\varepsilon(n)\hookrightarrow F$ are  epimorphisms of rings for
each $n$. Suppose now that we are in the situation of the foregoing
paragraph. By Lemma~\ref{lem:automorphism}(ii), $\sigma(x)$ can be
extended to  $F_\varepsilon(n)$ for each $x\in G$ and $n\geq
0$. Thus we obtain the embeddings $$SG\hookrightarrow
F_\varepsilon(n)G\hookrightarrow FG$$ for each $n\geq 0$. If,
moreover, $(G,<)$ is an ordered group, we get the embeddings of
Malcev-Neumann series rings $$S((G,<))\hookrightarrow
F_\varepsilon(n)((G,<))\hookrightarrow F((G,<))$$ for each $n\geq
0$.

Next result is a general version for Malcev-Neumann series of Proposition \ref{prop:heightskewpolynomial}.

\begin{theo}\label{theo:inversionheightseries}
Let $S$ be a domain with a field of fractions $\varepsilon\colon
S\hookrightarrow F$. Let $(G,<)$ be an ordered group. Consider a
crossed product $SG$ such that it can be extended to  a crossed
product $FG$. Let $E=F((G,<))$ be the associated Malcev-Neumann
series ring and $\iota\colon SG\hookrightarrow E$ be the natural
embedding. Then
\begin{enumerate}[\rm(i)]
\item  $E_\iota(n)\subseteq \mathcal{L}_n=F_\varepsilon(n)((G,<))$
for each integer $n\geq 0$.
\item Let $f\in F$. If $\h_\varepsilon(f)=n$, then $\h_\iota(f)=n$.
\item $\h_\iota(SG)\geq \h_\varepsilon(S)$.
\end{enumerate}
\end{theo}

\begin{proof}
We prove (i) by induction on $n$. For $n=0$ the result is clear
because $E_\iota(0)=SG\subseteq \mathcal{L}_0$.   So suppose that
(i) holds for $n\geq 0,$ and we must prove it for $n+1.$

By the definition of $E_\iota(n+1)$, and the fact that
$\mathcal{L}_{n+1}$ is a ring, it suffices to prove that if $f\in
E_\iota(n)\setminus\{0\}$ then $f^{-1}\in  \mathcal{L}_{n+1}$. By
induction hypothesis, $f\in\mathcal{L}_n$. Suppose that
$f=\sum\limits_{x\in G}a_x\bar{x}$ with $a_x\in F_\varepsilon(n).$
Let $x_0=\min\{x\in G\mid x\in\supp f\}.$ Then,
$$f^{-1}=(a_{x_0}\bar{x}_0)^{-1}\sum_{m\geq0}\big(g(a_{x_0}\bar{x}_0)^{-1}\big)^m,$$
where $g=a_{x_0}\bar{x}_0-f\in \mathcal{L}_n.$  Note that
$(a_{x_0}\bar{x}_0)^{-1}=\bar{x}_0^{-1}a_{x_0}^{-1}=
\bar{x}_0^{-1}a_{x_0}^{-1}\bar{x}_0\bar{x}_0^{-1}=
{}^{\sigma(x_0)^{-1}}(a_{x_0}^{-1})\bar{x}_0^{-1}\in
\mathcal{L}_{n+1}$, and thus $g(a_{x_0}\bar{x}_0)^{-1}\in
\mathcal{L}_{n+1}.$ Since $\mathcal{L}_{n+1}$ is a ring,
$(g(a_{x_0}\bar{x}_0)^{-1})^m\in \mathcal{L}_{n+1}$ for each
$m\geq0.$  By \S \ref{subsec:crossedproducts}, the series $\sum\limits_{m\geq
0}(g(a_{x_0}\bar{x}_0)^{-1})^m$ is well defined in $E$. Hence,
for each $x\in G$, the coefficient of $\bar{x}$ in
$\sum\limits_{m\geq 0}(g(a_{x_0}\bar{x}_0)^{-1})^m$ is an element of
$F_\varepsilon(n+1)$, i.e. $\sum\limits_{m\geq
0}(g(a_{x_0}\bar{x}_0)^{-1})^m\in\mathcal{L}_{n+1}$. Therefore
$$f^{-1}=(a_{x_0}\bar{x}_0)^{-1}\sum\limits_{m\geq
0}(g(a_{x_0}\bar{x}_0)^{-1})^m\in \mathcal{L}_{n+1}.$$

(ii) If $S$ is a field, the result is clear. So suppose that $S$ is
not a field. Let $f\in F_\varepsilon(n+1)\setminus F_\varepsilon(n)$
for some integer $n\geq 0$. Since $S\subseteq SG$, clearly $f\in
F_\varepsilon(n+1)\subseteq E_\iota(n+1).$ Suppose that
$\h_\iota(f)\leq n$, that is, $f\in
E_\iota(n)\subseteq\mathcal{L}_n$. By (i), $f=\sum_{x\in
G}a_x\bar{x}$ with $a_x\in F_\varepsilon(n)$.
 Observe that two series
$\sum_{x\in G} b_x\bar{x},\ \sum_{x\in G} c_x\bar{x}\in E$, where
$b_x,c_x\in F$ for each $x\in G$, are equal if and only if $b_x=c_x$
for each $x\in G.$ Hence $f=a_1\in F_\varepsilon(n)$, a
contradiction. Hence $\h_\iota(f)=n+1$.

(iii) Follows from (ii).
\end{proof}

If $G$ is a group and $x,y\in G$, by $(x,y)$ we denote the
commutator $(x,y)=x^{-1}y^{-1}xy$.

It is well known that a torsion-free nilpotent group is orderable.
Also, the free product of orderable groups is orderable. Hence, if
we are given a set of torsion-free nilpotent groups $\{G_i\}_{i\in
I}$, the free product $\ast_{i\in I}G_i$ is an orderable group.

\begin{coro}
Let $k$ be a commutative field, $I$ be a set of cardinality at least
two and $\{G_i\}_{i\in I}$ be a set of torsion-free nilpotent
groups. Set
$G=\operatornamewithlimits{\ast}\limits_{\scriptscriptstyle i\in
I}G_i$, and suppose that $(G,<)$ is an ordered group. Let $k[G]$ be
the group ring  and $\iota\colon k[G]\hookrightarrow E=k((G,<))$ be
the natural embedding  in its  Malcev-Neumann series ring. Then
$\h_\iota(k[G])=\infty.$ Indeed, let $x\in G_i\setminus\{1\}$ and
$y\in G_j\setminus\{1\}$ with $i\neq j$. If $f$ is any entry of the
inverse of the $n\times n$ matrix
$$A_n=\left(\begin{array}{cccc} (x,y) & (x,y^2) &
\cdots & (x,y^n) \\ (x^2,y) & (x^2,y^2) & \cdots &
(x^2,y^n)\\ \vdots & \vdots & \cdots & \vdots \\
(x^n,y) & (x^n,y^2) & \cdots & (x^n,y^n)
\end{array}\right),$$
then $\h_\iota(f)=n$.

In particular, if $X$ is a set of cardinality at least two and $G$
is the free group on $X$, then the universal field of fractions
$\iota'\colon k[G]\hookrightarrow F$ and $\iota\colon \freealgebra
kX\hookrightarrow F$ are of infinite inversion height. Indeed, let
$x,y\in X$ be different elements, if $f$ is any entry of the inverse
of the $n\times n$ matrix,
$$A_n=\left(\begin{array}{cccc} (x,y) & (x,y^2) &
\cdots & (x,y^n) \\ (x^2,y) & (x^2,y^2) & \cdots &
(x^2,y^n)\\ \vdots & \vdots & \cdots & \vdots \\
(x^n,y) & (x^n,y^2) & \cdots & (x^n,y^n)
\end{array}\right),$$
then $\h_{\iota'}(f)=n-1$ and $\h_\iota(f)=n$.
\end{coro}

\begin{proof}
Consider $\bigoplus_{i\in I}G_i$, the subgroup of the cartesian
product $\prod\limits_{i\in I}G_i$ consisting of all $(x_i)_{i\in
I}\in \prod\limits_{i\in I}G_i$ such that $x_i=1$ for almost all
$i\in I$.

For each $i\in I$, let $\pi_i\colon G_i\hookrightarrow
\bigoplus_{i\in I}G_i$ be the canonical inclusion and let
$\pi\colon\ast_{i\in I}G_i\rightarrow \bigoplus_{i\in I}G_i$ be the
unique morphism of groups such that $\pi_{\mid G_i}=\pi_i$. Set
$N=\ker \pi$, then $N$ is a free group. Since the
cardinality of $I$ is at least two, and each $G_i$ is an infinite
group for each $i$, $N$ is not finitely generated. Indeed, if we fix
a total order $\prec$ on $I$, then $N$ is the free group on the
nontrivial elements of the set
$$\Bigl\{(x_{i_1}x_{i_2}\cdots x_{i_r},x_{i_{r+1}}\cdots
x_{i_s})\mid x_{i_j}\in G_{i_j},\ i_1\prec i_2\prec\dotsb\prec
i_s\in I\Bigr\}.$$ Hence $G$ is the extension of the free group $N$
by the group $G/N\cong \bigoplus_{i\in I}G_i$. Recall that since
$G/N$ is locally nilpotent, any crossed product $F\frac{G}{N}$, with
$F$ a field, is an Ore domain.

If $\varepsilon=\iota_{\mid k[N]}\colon k[N]\hookrightarrow
E_\varepsilon$, then $\varepsilon$ is the universal field of
fractions of $k[N]$. Any automorphism of $k[N]$ can be extended to
$E_\varepsilon$ by
\cite[Corollary~7.5.16]{Cohnfreeeidealringslocalization}. Hence the
crossed product $k[N]\frac{G}{N}$ extends to $E_\varepsilon
\frac{G}{N}$. Another way of proving this extension can be found in
\cite[Proposition~2.5(1)]{SanchezfreegroupalgebraMNseries}.

If for each $\alpha\in\frac{G}{N}$, we pick a coset representative
$x_\alpha\in G$, then the set $\{x_\alpha\}_{\alpha\in\frac{G}{N}}$
is linearly independent over $E_\varepsilon$ (in fact over
$k((N,<))$). Thus $E_\varepsilon\frac{G}{N}\hookrightarrow E$. Hence
$k[G]\hookrightarrow E_\varepsilon \frac{G}{N}\hookrightarrow
E_\iota$. The crossed product $E_\varepsilon\frac{G}{N}$ is an Ore
domain and $E_\varepsilon\frac{G}{N}\hookrightarrow E_\iota$ is the
Ore field of fractions of $E_\varepsilon\frac{G}{N}$. By
Theorem~\ref{theo:inversionheightseries}(iii),
$\h_\iota(k[G])=\h_{\iota}(k[N]\frac{G}{N})\geq
\h_\varepsilon(k[N])$, and $\h_\varepsilon(k[N])=\infty$ by
Corollary~\ref{coro:matrixheight}. Also, by
Theorem~\ref{theo:inversionheightseries}(ii) and
Corollary~\ref{coro:matrixheight}, $\h_\iota(f)=n$.

The fact that $\h_{\iota}(\freealgebra kX)=\h_{\iota'}(k[G])=\infty$
follows from the fact that a free group is a free product of
infinite cyclic groups, and because we can identify the free field inside the Malcev-Neumann power series ring cf. \S \ref{localitzacio}. That $\h_\iota(f)=n$ and
$\h_{\iota'}(f)=n-1$ follows from Corollary~\ref{coro:matrixheight}.
\end{proof}

\begin{prop}\label{prop:inftyinftyembeddingsinftyinversionheight}
Let $k$ be a commutative field. For each finite set $X$ with
$|X|\geq 2,$ there exist infinite non-isomorphic fields of fractions
$\iota\colon\freealgebra kX\rightarrow D$ such that
$\h_\iota(\freealgebra kX)=\infty.$
\end{prop}

\begin{proof}

\underline{Step 1:} We define a poly-orderable group $\Gamma_r$ for
each integer $r\geq 1$.

We follow the notation in \cite[Chapter~1]{WarrenDunwoody}. Let
$r\geq 1$. Let $Y$ be the connected graph with vertex set
$VY=\mathbb{Z}$, edge set $EY=\{e_i\mid i\in\mathbb{Z}\}$ and
incidence functions $\bar{\iota}(e_i)=i$ and $\bar{\tau}(e_i)=i+1$,
i.e.
$$\xymatrix@C=1.2cm{\dotsb  \ar[r]^{\scriptscriptstyle e_{i-1}} & \stackrel{
i}{\bullet}\ar[r]^{\scriptscriptstyle e_i} & \stackrel{
i+1}{\bullet}\ar[r]^{\scriptscriptstyle e_{i+1}}& \dotsb}$$ Let
$(G(\ ),Y)$ be the graph of groups $$\xymatrix@C=1.2cm{\dotsb \ar[r]
& \stackrel{ G(i)}{\bullet}\ar[r]^{\scriptscriptstyle G(e_i)} &
\stackrel{ G(i+1)}{\bullet}\ar[r]^{\scriptscriptstyle G(e_{i+1})}&
\dotsb}$$ where $G(i)$ is the free abelian group on
$\{T_i,T_{i+1},\dotsc,T_{i+r}\}$ and $G(e_i)$ the free abelian group
on $\{T_{i+1},\dotsc,T_{i+r}\}$ for each $i\in\integers.$ Let $N_r$
be the fundamental group of $(G(\ ),Y)$, i.e. $N_r=\pi(G(\ ),Y,Y_0)$
with $Y_0=Y$. Then, by definition,
\begin{equation}\label{eq:definitionNr}
N_r=\left\langle T_i,\ i\in\integers \left\lvert
\begin{smallmatrix} T_iT_{i+1}=T_{i+1}T_i\\
T_iT_{i+2}=T_{i+2}T_i\\ \cdots
\\ T_iT_{i+r}=T_{i+r}T_i
\end{smallmatrix}\right.\right\rangle.
\end{equation}
Also $N_r$ can be seen as
\begin{equation}\label{eq:treeproductN_r} \cdots
G(i-1)\ast_{G(e_{i-1})}G(i)\ast_{G(e_i)}G(i+1)\ast_{G(e_{i+1})}\cdots.
\end{equation}
Consider the morphism of groups $\theta\colon N_r\rightarrow
\bigoplus_{i\in\integers}\integers$ defined by $\theta(T_i)=f_i$
where $f_i$ is the sequence $(x_n)_{n\in\integers}$ with $x_i=1$ and
$x_n=0$ for $n\neq i$. It is easy to deduce from
\eqref{eq:definitionNr} that $\theta$ is well defined. Let $L_r=\ker
\theta$. Observe that $\theta_{\mid G(i)}$ is injective for each
$i\in\integers$. Hence $L_r$ is a free group by
\cite[Proposition~7.10]{WarrenDunwoody}. Moreover, $L_r$ is not
commutative because for example $T_0T_{r+1}T_0^{-1}T_{r+1}^{-1}$ and
$T_{2r+2}T_{3r+3}T_{2r+2}^{-1}T_{3r+3}^{-1}$ belong to $L_r$, but
they do not commute as can be deduced from
\eqref{eq:treeproductN_r}. In a similar way, it can be seen that
$L_r$ is not finitely generated.

Define now $\Gamma_r=N_r\rtimes C$, where $C=\langle S\rangle$ is
the infinite cyclic group, and $C$ acts on $N_r$ as $T_i\mapsto
T_{i+1}$, i.e. $ST_iS^{-1}=T_{i+1}$. Hence $\Gamma_r$ has the
subnormal series
$$1\lhd L_r\lhd N_r\lhd \Gamma_r,$$
with $\Gamma_r/N_r=C$ infinite cyclic, $N_r/L_r\cong
\bigoplus_{i\in\integers}\integers$ a torsion-free abelian group and
$L_r$ a noncommutative free group. Hence all factors are orderable
groups.

\medskip

\noindent\underline{Step 2:} We construct a field of fractions
$\delta_r\colon k[\Gamma_r]\hookrightarrow E^r$ with
$\h_{\delta_r}(k[\Gamma_r])=\infty$ for each integer $r\geq 1$.

Let $r\geq 1$. Consider $\beta_r\colon k[L_r]\hookrightarrow C_r$
the universal field of fractions of the free group algebra $k[L_r]$.
Consider $k[N_r]$ as a crossed product $k[L_r]\frac{N_r}{L_r}$. Any
automorphism of $k[L_r]$ can be extended to an automorphism of $C_r$
by \cite[Corollary~7.5.16]{Cohnfreeeidealringslocalization}. Hence
we can consider a crossed product $C_r\frac{N_r}{L_r}$ that contains
$k[L_r]\frac{N_r}{L_r}$ in the natural way. Since $N_r/L_r$ is a
torsion-free abelian group, $C_r\frac{N_r}{L_r}$ is an Ore domain.
Let $\gamma_r\colon k[N_r]\hookrightarrow
C_r\frac{N_r}{L_r}\hookrightarrow D_r$ where $D_r$ is the Ore field
of fractions of $C_r \frac{N_r}{L_r}$. The group ring $k[\Gamma_r]$
can be seen as a skew Laurent polynomial ring
$k[N_r][S,S^{-1};\alpha]$ where $\alpha$ is given by left
conjugation by $S$ cf. \S \ref{subsec:crossedproducts}. Observe that
conjugation by $S$ induces an automorphism on $L_r$, thus on
$k[L_r]$ and on $C_r$. Therefore it can be extended to an
automorphism of $C_r \frac{N_r}{L_r}$. Since $C_r \frac{N_r}{L_r}$
is an Ore domain and conjugation by $S$ gives an automorphism of
$C_r \frac{N_r}{L_r}$, it can be extended to $D_r$. Hence
$$k[\Gamma_r]= k[N_r][S,S^{-1};\alpha]\hookrightarrow C_r \frac{N_r}{L_r} [S,S^{-1};\alpha]
\hookrightarrow D_r[S,S^{-1};\alpha].$$ Let $E^r$ be the Ore field
of fractions of $D_r[S,S^{-1};\alpha]$, and $\delta_r\colon
k[\Gamma_r]\hookrightarrow E^r$ be the natural embedding. Observe
that it is a field of fractions of $k[\Gamma_r]$.

For $C_r\frac{N_r}{L_r}$ and $D_r[S,S^{-1};\alpha]$ are Ore domains,
we can think of $E^r$ and $D_r$ as embedded in $D_r((S;\alpha))$ and
$C_r((\frac{N_r}{L_r},<))$ for a certain order $<$ of $N_r/L_r$,
respectively. Now by
Proposition~\ref{prop:heightskewpolynomial2}(iii),
$\h_{\delta_r}(k[\Gamma_r])\geq \h_{\gamma_r}(k[N_r])$. By
Theorem~\ref{theo:inversionheightseries}(iii), $\h_{\gamma_r}(k
[N_r])\geq \h_{\beta_r}(k[L_r])$. By
Corollary~\ref{coro:matrixheight}, $\h_{\beta_r}(k[L_r])=\infty$.
Therefore $\h_{\delta_r}(k[\Gamma_r])=\infty$.

\medskip

\noindent\underline{Step 3:} For each pair of integers $1\leq r\leq
s$, the free algebra $\freealgebra k {X_0,X_1,\dotsc,X_r}$ embeds in
$k[\Gamma_s]$ via $X_i\mapsto T_0^iS$.

 Let $K=k(t)$ be the field of fractions of the polynomial ring
$k[t]$. Let $\alpha_s\colon K\rightarrow K$ be the morphism of rings
given by $\alpha_s(t)=t^{s+1}$. Consider the skew polynomial ring
$R_s=K[x;\alpha_s]$ and let $F_s$ be the Ore field of fractions of
$R_s$. Then $\{1,t,\dotsc,t^s\}$ are right linearly independent over
$k(t^{s+1})=\alpha_s (K)$. Then $\{x,tx,\dotsc,t^sx\}$ are right
linearly independent over $R_s$ \cite[Lemma~9.19]{Lam2}. By
Jategaonkar's Lemma \cite[Lemma~9.21]{Lam2}, there is an embedding
of rings $\freealgebra k{X_0,\dotsc,X_r}\hookrightarrow
R_s\hookrightarrow F_s$ sending $X_i\mapsto t^ix$. Consider now the
morphism of groups $\varepsilon\colon \Gamma_r\rightarrow
F_s^\times$ defined by $\varepsilon (T_i)=t^{(s+1)^i}$ and
$\varepsilon(S)=x$. Since $\varepsilon(T_0^iS)=t^ix$ for $1\leq
i\leq r$, the set $\{S,T_0S,\dotsc,T_0^rS\}$ is a basis of a free
monoid inside $\Gamma_r$. Therefore we obtain the embedding
$\freealgebra k{X_0,\dotsc,X_r}\hookrightarrow k[\Gamma_s]$ of
$k$-algebras defined by $X_i\mapsto T_0^iS$.

\medskip

\noindent\underline{Step 4:} For each pair of integers $1\leq r\leq
s$, there is a field of fractions $\iota_{rs}\colon \freealgebra k
{X_0,X_1,\dotsc,X_r}\hookrightarrow E^s$, defined by
$\iota_{rs}(X_i)= T_0^iS$, of infinite inversion height.

Let $\iota_{rs}$ be the embedding of Step~3 composed with
$\delta_s$. Note that $S=\iota_{rs}(X_0)$ and
$T_0=\iota(X_1)\iota(X_0)^{-1}$. Since $\Gamma _s$ is generated by $S$ and $T_0$,  $k[\Gamma_s]\subseteq
E_{\iota_{rs}}^s(2)$. Therefore $E^s$ is generated, as a field, by the
image of $\iota_{rs}$ and $\h_{\iota_{rs}}(\freealgebra
k{X_0,\dotsc,X_r})\geq \h_{\delta_s}(k[\Gamma_s])=\infty$.

\medskip

\noindent\underline{Step 5:} The fields of fractions
$\iota_{rs}\colon \freealgebra k {X_0,X_1,\dotsc,X_r}\hookrightarrow
E^s$ and $\iota_{rs'}\colon \freealgebra k
{X_0,X_1,\dotsc,X_r}\hookrightarrow E^{s'}$ are not isomorphic for
$s\neq s'$.

Let $1\leq r\leq s<s'$ be integers. First of all observe that there
does not exist an isomorphism of groups $\Gamma_s\rightarrow
\Gamma_{s'}$ with $S\mapsto S$ and $T_0\mapsto T_0$.

If there is an isomorphism of rings $\eta_{ss'}\colon E^s\rightarrow
E^{s'}$ such that $\iota_{rs'}=\eta_{ss'}\iota_{rs}$, then
$$\eta_{ss'}(S)=\eta_{ss'}(\iota_{rs}(X_0))=\iota_{rs'}(X_0)=S,$$
$$\eta_{ss'}(T_0)=\eta_{ss'}(\iota_{rs}(X_1)\iota_{rs}(X_0)^{-1})=\iota_{rs'}(X_1)\iota_{rs'}(X_0)^{-1}=T_0.$$
Hence the restriction of $\eta_{ss'}{}_{\mid\Gamma_s}\colon
\Gamma_s\rightarrow \Gamma_{s'}$ gives an isomorphism of groups
sending $S\mapsto S$ and $T_0\mapsto T_0$, a contradiction.
\end{proof}

\begin{coro}
Let $k$ be a commutative field and $Z=\{z_1,z_2,\dotsc\}$ be an
infinite countable set. Then the free algebra $\freealgebra kZ$ has
infinite non-isomorphic fields of fractions $\iota\colon
\freealgebra kZ\rightarrow D$ such that $\h_\iota(\freealgebra
kZ)=\infty$.
\end{coro}

\begin{proof}
Follows from \cite[Proposition~2.3]{Herberasanchez} and
Proposition~\ref{prop:inftyinftyembeddingsinftyinversionheight}.
\end{proof}


\section{Crossed products of a ring by a universal enveloping
algebra}\label{sec:crossedproducts}

Throughout this section, $k$ will denote a commutative field.

Let $L$ be a Lie $k$-algebra. We will denote by $U(L)$ its
\emph{universal enveloping algebra}.
Suppose that $R$ is a $k$-algebra, and let $\Der_k(R)$ denote the
set of $k$-linear derivations of $R$. A $k$-algebra $S$ containing
$R$ is called \emph{crossed product} of $R$ by $U(L)$ (and written
$R*U(L)$) provided that there is a $k$-linear embedding $^-\colon
L\rightarrow S$, $x\mapsto \bar{x}$, such that:
\begin{enumerate}[(i)]
\item $S$ has the additive structure of $R\otimes_k U(L)$.

\item There exist a $k$-linear  map (called \emph{action}) $\delta\colon L\rightarrow
\Der_k(R)$, $x\mapsto \delta_x$, and a $k$-bilinear antisymmetric map (called
\emph{twisting}) $t\colon L\times L\rightarrow R$, $(x,y)\mapsto
t(x,y)$ such that the following two conditions hold:
\begin{eqnarray}
\bar{x}a=a\bar{x}+\delta_x(a)\quad \textrm{for all } x\in L \textrm{
and } a\in R, \label{def:crossedproduct1}\\
\bar{x}\bar{y}-\bar{y}\bar{x}=\overline{[x,y]}+t(x,y)\quad
\textrm{for all } x,y\in L. \label{def:crossedproduct2}
\end{eqnarray}
\end{enumerate}

Crossed products for Lie Algebras were introduced in
\cite[1.7.12]{McConnellRobson} and in \cite{Chin}.

\smallskip

Let $\mathcal{C}$ be a $k$-linear independent subset of $L$. Suppose
that we have defined a total order $<$ in $\mathcal{C}$. The
\emph{standard monomials} in $\mathcal{C}$ is the subset of $R*U(L)$
consisting on the monomials of the form
${\bar{x}_1}\bar{x}_2\dotsb\bar{x}_m$ with $m\geq0$,
$x_i\in\mathcal{C}$ and $x_1\leq x_2\leq \dotsb\leq x_m$ where we
understand that the identity element in $U(L)$ is the standard
monomial corresponding to $m=0$. A standard monomial that is the
product of $m$ elements of $\mathcal{C}$ has degree $-m$.

Let $\mathcal{B}=\{x_i\mid i\in
I\}$ be a totally ordered basis of $L$. The
Poincar\'e-Birkhoff-Witt (PBW) Theorem states that the standard monomials in $\mathcal{B}$
form a $k$-basis of $U(L)$. Thus (i) above is equivalent to the fact that $R*U(L)$ is a free left
$R$-module with basis the standard monomials in $\mathcal{B}$.

One of the most important properties of crossed products is the
following result  which is \cite[Lemma~1.1]{BergenPassman}.  We will need how the identification
\eqref{lem:propertycrossedproducts} is made, thus we sketch the
proof of \cite[Lemma~1.1]{BergenPassman}.

\begin{lem}\label{lem:importantpropertycrossedproduct}
If $H$ is an ideal of $L$, then
\begin{equation}\label{lem:propertycrossedproducts}
R*U(L)=(R*U(H))*U(L/H).
\end{equation}
\end{lem}

\begin{proof}
 Set $T=R*U(H)$. Let $W$ be
a subspace of $L$ with $L=H\oplus W$ and let $\sigma\colon
L/H\rightarrow W$ be a $k$-vector space isomorphism. Let
$\mathcal{D}$ be an ordered basis for $L/H$ and let $\mathcal{C}$ be
one for $H$. Then $\mathcal{C}\cup \{\sigma(d)\mid
d\in\mathcal{D}\}$ is an ordered basis for $L$ with the elements of
$\mathcal{C}$ coming first. Then $R*U(L)$ has the additive structure
of $T*U(L/H)$ by the PBW-theorem. Let $\tilde{\phantom{a}}$ denote
the composition of $\sigma$ followed by $^-$. Then, for each $x\in
L/H$ and $t\in T$, we have that $\zeta_x(t)=t\tilde{x}-\tilde{x}t\in
T$. Thus we get a $k$-linear map $\zeta\colon L/H\rightarrow
\Der_k(T)$, $x\mapsto \zeta_x$. Also it is not very difficult to see
that, for each $x,y\in L/H$,
$s(x,y)=\tilde{x}\tilde{y}-\tilde{y}\tilde{x}-\widetilde{[x,y]}\in
T$. We thus define the $k$-linear map $s\colon L/H\times
L/H\rightarrow T$, $(x,y)\mapsto s(x,y)$.
\end{proof}

For a given embedding of rings $R\hookrightarrow D$, we will be
interested in extending the crossed product structure of $R*U(L)$ to
$D*U(L)$ in the natural way. In order to do that we need to be
precise on the conditions that $\delta$ and $t$ must satisfy. This
is explained in the next lemma which can be seen as a corollary of
\cite[Theorem 7.1.10]{Montgomery}

\begin{lem}\label{lem:conditionsforcrossedproduct}
Let $R$ be a $k$-algebra, and let $L$  be a Lie $k$-algebra. Suppose
that there exist a $k$-linear map $\delta\colon L\rightarrow
\Der_k(R)$, $x\mapsto \delta_x$, and a $k$-bilinear antisymmetric map
$t\colon L\times L\rightarrow R$, $(x,y)\mapsto t(x,y)$. They define
a crossed product $R*U(L)$ if and only if $\delta$ and $t$ satisfy
the following relations:
\begin{enumerate}[\rm(i)]
\item
$\delta_x(t(y,z))+\delta_y(t(z,x))+\delta_z(t(x,y))+t(x,[y,z])+t(y,[z,x])+t(z,[x,y])=0.$
\item $[\delta_x,\delta_y]=\delta_{[x,y]}+\partial_{t(x,y)}$ where
$\partial_{t(x,y)}$ denotes the $k$-derivation of $R$ defined by
$a\mapsto [t(x,y),a]=t(x,y)a-at(x,y)$ for all $a\in R$.
\end{enumerate}

Moreover, $R*U(L)$ can be constructed as the $k$-coproduct of $R$ with $T(L)$, the $k$-tensor algebra over $L$, modulo the two-sided ideal $\mathcal{I}$ generated by the set
\[\{xa-ax-\delta _x(a), \quad xy-yx-[x,y]-t(x,y)\mid \mbox{ for any $x,y\in L$ and $a\in R$}\}, \]
and it is free as a right and as a left $R$-module. More precisely, if
  $\mathcal{B}=\{e_j\mid j \in
J\}$ is a fixed ordered basis for $L$,  then the set $\mathcal{G}$
of standard monomials on $\mathcal{B}$ is a basis of $R*U(L)$ as a
right and as a left $R$-module;  and if, for any $m\ge 0$,
$\mathcal{G}_m\subseteq \mathcal{G}$ denotes the set of standard
monomials of degree at most $m$, then $\sum _{x\in \mathcal{G}_m}xR
= \sum _{x\in \mathcal{G}_m}Rx$. \qed
\end{lem}

\begin{rem}\label{rem:extensioncrossedproduct} Let $f\colon R\hookrightarrow
D$ be an  extension of $k$-algebras, and let $L$ be a Lie
$k$-algebra such that there exists a crossed product $R*U(L)$. To
extend the crossed product structure to a crossed product $D*U(L)$
in such a way there is a ring  inclusion $\tilde{f}\colon
R*U(L)\hookrightarrow D*U(L)$ extending $f$ and such that
$\tilde{f}(\overline{x})=\overline{x}$, for any $x\in L$, one has:
\begin{enumerate}[(1)]
\item to make sure that the standard monomials are left
$D$-independent;
\item to extend the action $\delta _R$ to a $k$-linear map
$\delta_D\colon L\to \mathrm{Der}_k(D)$ in such a way that, for any
$r\in R$, $\delta _R(x)(r)=\delta _D(x)(f(r))$;
\item to make sure that
condition (ii) in  Lemma \ref{lem:conditionsforcrossedproduct} is satisfied.
\end{enumerate}
Notice that the twisting must be the same for both crossed
products, so that it is not necessary to verify condition (i)
in Lemma~\ref{lem:conditionsforcrossedproduct}.

Usually, we will be working with ring embeddings such that the
derivations over $R$ extend in a unique way to $D$ (as in Lemma
\ref{lem:extensionofderivations}), so that   condition (2)
above will be automatically satisfied. Hence, only conditions (1) and (3) above need to be verified. \qed
\end{rem}

The existence of a PBW-basis for $R*U(L)$,  asserted in  Lemma
\ref{lem:conditionsforcrossedproduct}, gives a structure of filtered
ring to $R*U(L)$ by setting, for any $m\ge 0$, $\mathcal{F}_m$ to be
the $R$-subbimodule of $R*U(L)$ generated by the monomials of
degree at most $m$.  By the definition of crossed product and Lemma
\ref{lem:conditionsforcrossedproduct}, the associated graded ring is
a polynomial ring over $R$ in the commutative variables given by the
basis of the Lie algebra $L$. For further quoting we summarize this
in the next Lemma.

\begin{lem}\label{lem:graded ring}
Let $R$ be a $k$-algebra, and let $L$ be a Lie $k$-algebra.
Suppose that there exists a crossed product $R*U(L)$. Fix $\mathcal{B}$ to be a basis of $L$,
then $\mathrm{gr} (R*U(L))\cong R[\mathcal{B}]$, that is, a polynomial algebra over $R$ in the
commuting variables $\mathcal{B}$. \qed
\end{lem}

In the foregoing lemma, if $R$ is a field, then  $\mathrm{gr} (R*U(L))$ is an Ore domain, which implies
that $R*U(L)$ embeds in a field with some good properties.  This is  expressed more procisely in the next
proposition.

\begin{prop}\label{prop:canonicalfieldoffractions}
Let $L$ be a Lie $k$-algebra and $K$ be a field with $k$ as a central subfield.
For each crossed product $K*U(L)$, there is a canonically
constructed field of fractions $$K*U(L)\hookrightarrow
\mathfrak{D}(K*U(L)).$$ Suppose that $N$ is a subalgebra of $L$.
The following properties are satisfied:
\begin{enumerate}[\rm (i)]
\item The following diagram is
commutative $$\xymatrix{ K*U(N)\ar@{^{(}->}[r] \ar@{^{(}->}[d] &
\mathfrak{D}(K*U(N))\ar@{^{(}->}[d]\\
K*U(L) \ar@{^{(}->}[r] & \mathfrak{D}(K*U(L))}$$
\item If $\mathcal{B}_N$ is a basis of $N$ and $\mathcal{C}$
is a set of elements of $L\setminus N$ such that $\mathcal{B}_N\cup \mathcal{C}$  is
a basis of $L$, then the
standard monomials in $\mathcal{C}$
 are linearly independent over
$\mathfrak{D}(K*U(N))$.
\item If  $N$ is an ideal of $L$, then the subring of $\mathfrak{D}(K*U(L))$ generated
by $K*U(L)$ and $\mathfrak{D}(K*U(N))$ is a crossed product $\mathfrak{D}(K*U(N))*U(L/N)$ extending
 $(K*U(N))*U(L/N)$ in the natural way.
\end{enumerate}
\end{prop}

\begin{proof}
By Lemma~\ref{lem:graded ring}, $\gr(K*U(L))$ is an Ore domain. Now
\cite[Theorem~2.6.5]{Cohnskew} or \cite{Lichtmanvaluationmethods}
imply the existence of the construction $K*U(L)\hookrightarrow
\mathfrak{D}(K*U(L))$.

Conditions (i) and (ii) can be proved in exactly the same way as for the embedding  $U(L)\hookrightarrow \mathfrak{D}(U(L))$,
see  \cite[Proposition~2]{Lichtmanuniversalfields}. Condition (iii) follows as in \cite[Section~2.3]{Lichtmanuniversalfields}.
\end{proof}

Now we turn our attention to crossed products where the underlying Lie algebra is free.

\begin{lem}\label{lem:crossedproductfir}
Let $R$ be a $k$-algebra. Let $H$
be the free Lie algebra on a set $X$. If $R*U(H)$ is a crossed
product then, for each $x\in X$, there exists a $k$-derivation
$\partial _x\colon R\rightarrow R$ such that $R*U(H)\cong
\coprod_{x\in X} R[x;\partial_x]$ the ring coproduct over $R$.

In particular, if $R=K$ is a field then $K*U(H)$ is a fir.
\end{lem}

\begin{proof}
Consider the Lie $k$-algebra structure of $R*U(H)$ where the Lie
product is given by $[a,b]=ab-ba$ for all $a,b\in R*U(H)$. Consider
the morphism of Lie $k$-algebras $\tilde{\phantom{s}}\colon
H\rightarrow K*U(H)$ which sends each $x\in X$ to $\bar{x}$. Thus
$\tilde{z}\tilde{w}-\tilde{w}\tilde{z}=\widetilde{[z,w]}$ for all
$z,w\in H$.

By induction on the length of the Lie words on $X$ and then
extending by linearity to $H$, it is not difficult to see that for
each $z\in H$,
\begin{equation}\label{eq:crossedproductcoproduct}
\tilde{z}=\bar{z}+b_z \textrm{ for some } b_z\in R.
\end{equation}
It is known that $U(H)$ is $\freealgebra kX$, the free $k$-algebra
on the set $X$. Thus
 $R*U(H)$ is a free
left $R$-module with basis the free monoid on the set $\{\bar{x}\mid
x\in X\}$. By \eqref{eq:crossedproductcoproduct}, it follows that
$R*U(H)$ is a free left $R$-module with basis the free monoid on the
set $\{\tilde{x}\mid x\in X\}$. Thus $R*U(H)$ has the same additive
structure as $R\otimes_k k\langle X\rangle=R\otimes_k U(H)$.

Also from \eqref{eq:crossedproductcoproduct}, it follows that
\begin{equation}\label{eq:crossedproductcoproduct2}
\tilde{z}a=a\tilde{z}+\partial_{z}(a)\quad \textrm{for each } a\in R
\textrm{ and } z\in H,
\end{equation}
where $\partial_{z}\in \Der_k(R)$ and is given by $a\mapsto
\delta_z(a)+[b_z,a]$. Thus we have just proved that $R*U(H)$ can be
thought as a crossed product with trivial twisting.

From \eqref{eq:crossedproductcoproduct2}, we deduce that, for each
$x\in X$, there exists a morphism of $R$-rings $$\varphi_x\colon
R[x;\partial_x]\rightarrow R*U(H)$$ which sends
$x\mapsto \tilde{x}$. Consider now the unique morphism of $R$-rings
$$\varphi\colon \coprod_{x\in X}R[x;\partial_x]\rightarrow R*U(H)$$
extending all $\varphi_x$. Proceeding as in \cite[Section
4]{bergmancoproducts} it is possible to prove that  the free monoid
on $X$ is a right and left $R$-basis of $\coprod_{x\in
X}R[x;\partial_x]$. Thus $\varphi$ is an isomorphism.

The statement when $R$ is a field follows from \cite[\S
6]{Cohnweakalgorithm}.
\end{proof}

Hence for a free Lie algebra $H$ and crossed product $K*U(H)$,
Lemma~\ref{lem:crossedproductfir} implies the existence of the
universal field of fractions of $K*U(H)$ and
Proposition~\ref{prop:canonicalfieldoffractions} the existence of
$K*U(H)\hookrightarrow\mathfrak{D}(K*U(H))$. We will show in the
next section that both fields of fractions are in fact the same.
Next two results will be useful in proving this
assertion, for their proof it is important to keep in mind the results quoted in the
section~\ref{localitzacio}.

 The statement of the next lemma is a slight
generalization of \cite[Lemma~1]{LewinLewin} while the proof remains
essentially the same. In Proposition
\ref{prop:generalizationProp3.1Lichtman}, it will be helpful in
recognizing isomorphic fields of fractions.

\begin{lem}\label{lem:subringsofspecializations} Let $R$ be a ring.
Let $F$ and $L$ be epic $R$-fields, and $\rho$ an $R$-specialization
from $F$ to $L$. Suppose that $S$ is a subring of $F$ contained in
the domain of $\rho$. Denote by $F_S$ and $L_{\rho(S)}$ the
subfields of $F$ and $L$ generated by $S$ and $\rho(S)$,
respectively, and consider their induced structure of $S$-fields. If
$L_{\rho (S)}$ is an $S$-field with a minimal prime matrix ideal,
then $F_S$ is an $S$-field of fractions contained in the domain of
$\rho$, and so $\rho$ maps $F_S$ isomorphically onto $L_{\rho (S)}.$
\end{lem}

\begin{proof}
Let $F_0$ be the domain of $\rho$. Then $F$ and $L$ are
$F_0$-fields, via the inclusion $F_0\hookrightarrow F$ and via
$\rho\colon F_0\rightarrow L$, hence $\rho$ is an
$F_0$-specialization. Let $\Sigma$ be the set of matrices over $S$
that become invertible over $L_{\rho(S)}$. Recall that  there exists
a unique $S$-ring homomorphism $g\colon S_\Sigma\to L_{\rho(S)}$,
and this morphism happens to be onto (cf.
section~\ref{localitzacio}).

Each matrix of $\Sigma$ is invertible over $F_0$ because it is
invertible over its residue class field $F_0/\ker\rho\cong L$, thus
it is also invertible over $F$. The matrices of $\Sigma$ are also
invertible over $F_S$ because when considered as endomorphisms of
finite dimensional vector spaces over $F_S$ they are injective. By
the universal property of the localization, there exists a unique
morphism of $S$-rings $f\colon S_\Sigma \to F$ such
$f(S_\Sigma)\subseteq F_0\cap F_S$. Therefore we may consider the
morphism of $S$ rings $\rho \circ f\colon S_\Sigma \to L_{\rho(S)}$;
the uniqueness of $g$ implies that $g=\rho \circ f$. Since $g$ is
onto, we can deduce that $f$ induces  an onto $S$-morphism from a
subring of $F_S$ to $L_{\rho(S)}$. Therefore such $S$-morphism is a
specialization from
 $F_S$ to $L_{\rho(S)}$ and, by the
minimality of the prime matrix ideal of $L_{\rho(S)}$, it  must be an isomorphism between $F_S$ and $L_{\rho(S)}$.  Therefore, the image of $f$ is
exactly $F_S$, i.e. $F_S$ is contained in $F_0$.
\end{proof}

The following proposition is a generalization of
\cite[Lemma~3.1]{Lichtmanuniversalfields} to crossed products.

\begin{prop}\label{prop:generalizationProp3.1Lichtman}
Let $K$ be a field with $k$ as a central subfield. Let $H$ be a Lie
$k$-algebra and let $N$ be an ideal of $H$. Consider a crossed
product $K*U(H)$. Suppose that the following two conditions are
satisfied:
\begin{enumerate}[\rm (1)]
\item $K*U(H)$ has a universal field of fractions $K*U(H)\hookrightarrow E$.
\item $R=K*U(N)$ has a  prime matrix ideal $\mathcal{P}$ whose
localization $R_{\mathcal P}$ is a field of fractions of $R$.
\end{enumerate}
Then $K*U(H)=R*U(H/N)$ can be extended to a crossed product
structure $R_\mathcal{P}*U(H/N)$, the embedding
$K*U(H)\hookrightarrow E$ can be extended to
$R_\mathcal{P}*U(H/N)\hookrightarrow E$ and this embedding is the
universal field of fractions of $R_\mathcal{P}*U(H/N)$.
\end{prop}

\begin{proof}
First note that since $R_\mathcal{P}$ is a field of fractions,
$\mathcal{P}$ is a minimal prime matrix ideal (cf. \S \ref{localitzacio}).

We view $K*U(H)$ as $R*U(H/N)$. By
Lemma~\ref{lem:extensionofderivations}, for each $x\in H/N$, the
$k$-derivation $\delta_x$ of $R$ can be extended to $R_\mathcal{P}$.
We denote this  extension again by $\delta_x$.

We want to construct a crossed product $R_\mathcal{P}*U(H/N)$. For
that we see that the conditions of Lemma
\ref{lem:conditionsforcrossedproduct} are satisfied. The first one
is clearly satisfied because it is an equality in $R$. For the
second one, we have to verify the equality of two $k$-derivations of
$R_\mathcal{P}$.  Since this equality holds in $R$,  it also holds
over $R_\mathcal{P}$ because of
Lemma~\ref{lem:uniqueextensionderivation}.

Let $\mathcal{B}$ be a basis of $H/N$. By
Proposition~\ref{prop:canonicalfieldoffractions},
$R_\mathcal{P}*U(H/N)$ has a field of fractions
$R_\mathcal{P}*U(H/N)\hookrightarrow
D=\mathfrak{D}(R_\mathcal{P}*U(H/N))$. Clearly the restriction
$K*U(H)\hookrightarrow D$ is a field of fractions of $K*U(H)$. Thus
there exists a $K*U(H)$-specialization $\rho$ from $E$ to $D$. By
Lemma~\ref{lem:subringsofspecializations}, $\rho$ gives by
restriction an isomorphism between the subfield $E_N$ of $E$
generated by $R$ and $R_\mathcal{P}$. Moreover, the standard
monomials on $\mathcal{B}$ are linearly independent over $E_N$ in
$E$, because their images via $\rho$ are linearly independent over
$R_\mathcal{P}$ in $D$. Thus the subring of $E$ generated by $E_N$
and $\{\bar{x}\mid x\in H/N\}$ is a crossed product isomorphic to
$R_\mathcal{P}*U(H/N)$, because $\rho^{-1}\circ \delta$ and
$\rho^{-1}\circ \tau$ induce and action and twisting, respectively,
for this subring provided $\delta$ and $\tau$ are the action and the
twisting of $R_\mathcal{P}*U(H/N)$. Thus
$R_\mathcal{P}*U(H/N)\hookrightarrow E$ is a field of fractions of
$R_\mathcal{P}*U(H/N)$. To prove that it is the universal field of
fractions, observe that any $(R_\mathcal{P}*U(H/N))$-field
 is a $(K*U(H))$-field that contains the field
$R_\mathcal{P}$. By Lemma~\ref{lem:subringsofspecializations}, there
exists a $(K*U(H))$-specialization from $E$ that contains
$R_\mathcal{P}$, and thus, arguing as above, it also contains
$R_\mathcal{P}*U(H/N)$. Hence, such specialization  is also an
$(R_\mathcal{P}*U(H/N))$-specialization.
\end{proof}

Next corollary is relatively easy but it gives an idea of how \emph{weak} is the structure of crossed product.

\begin{coro}\label{coro:universalfieldforanyLiealgebra}
For each field $K$ with $k$ as a central subfield and each Lie
$k$-algebra $L$, there exists a field $D$ that contains $K$ and a
crossed product $D*U(L)$ that has a universal field of fractions.
\end{coro}

\begin{proof}
Let $X$ be a set of generators of $L$. Let $H$ be the free Lie
$k$-algebra on $X$. Consider the morphism of Lie algebras
$H\rightarrow L$ that is the identity on $X$, and let $N$ be the
kernel of this morphism. Note that $L\cong H/N$.

Consider a crossed product $K*U(H)$. Set $R=K*U(N)$. Then
$K*U(H)=R*U(L)$. By Lemma~\ref{lem:crossedproductfir}, $K*U(H)$ has
a universal field of fractions $E$. Since $N$ is also a free Lie
$k$-algebra, $R$ is a fir by Lemma~\ref{lem:crossedproductfir}. Thus
$R$ has a universal field of fractions and it is of the form
$R_\mathcal{P}$ where $\mathcal{P}$ is the prime matrix ideal
consisting of the nonfull matrices over $R$. By
Proposition~\ref{prop:generalizationProp3.1Lichtman}, there is a
crossed product $R_\mathcal{P}*U(L)$ and
$R_\mathcal{P}*U(L)\hookrightarrow E$ is its universal field of
fractions.
\end{proof}

Let $G$ be a group, and fix an isomorphism $G\cong H/N$ where $H$ is
a free group and $N$ is a normal subgroup (hence, it is a free
group) of $H$. Consider an ordering of $H$. For any field $K$
consider the group algebra $K[H]$. It was proved in
\cite[Proposition 2.5]{SanchezfreegroupalgebraMNseries}, that the
crossed product structure $K[H]=K[N]*H/N$ can be extended to
$K((N))*H/N$ and this is a subring of the Malcev-Neumann series
field $K((H))$. This result combined with the fact that the
universal field of fractions of $K[H]$ can be seen as a subring of
$K((H))$, allows us to prove a result analogous to
Corollary~\ref{coro:universalfieldforanyLiealgebra}  for the case of
groups. That is, for any field $K$ and any group $G$ there is a
field $D$ containing $K$ and a crossed product $D*G$ that has a
universal field of fractions.


\section{A field of fractions of a crossed product of a residually nilpotent Lie
algebra.}\label{sec:afieldoffractions}

Throughout this section, $k$ will denote a commutative field.

\medskip

In this section we present a ring of series introduced by A. I.
Lichtman in \cite{Lichtmanuniversalfields}. This ring of series
$K((H))$ is constructed from a crossed product $K*U(H)$ of a field
$K$ by $U(H)$ where $H$ is a residually nilpotent Lie algebra
satisfying the  Q-condition (see
Section~\ref{subsec:residuallynilpotentcase}). It will play the role
of the Malcev-Neumann series ring $K((G))$ constructed from a
crossed product $KG$ of a field $K$ by an ordered group $G$.

We will give a detailed exposition of the construction of the ring
of series for some reasons. First, we expect to clarify  and
generalize in some aspects the one given in
\cite{Lichtmanuniversalfields}. Secondly, in
Theorem~\ref{theo:LewintheoremforLiealgebras}  we prove that, for a free Lie algebra $H$, this power series ring contains the universal field of fractions of $K*U(H)$, this is an extension of   \cite[Theorem~1]{Lichtmanuniversalfields}. Having in mind the analogy between free groups and free Lie algebras, this result   can be seen as a counterpart of
Lewin's Theorem \cite{Lewin}. Moreover, as an application, we will  produce further examples of
elements with arbitrary inversion height.

The construction is divided in two parts. In
Section~\ref{subsec:nilpotentcase} we construct a ring of series for
a crossed product of a ring $R$ by $U(L)$ where $L$ is a nilpotent
Lie algebra. In Section~\ref{subsec:residuallynilpotentcase}, we
give the general construction using the preceding case. The main
idea of such construction is presented in the following argument,
which is a generalization of
\cite[Section~4]{Lichtmanuniversalfields}.

\medskip

Let $L$ be a Lie $k$-algebra, $R$ a $k$-algebra and $R*U(L)$ a crossed product.

Suppose that the \emph{center} of $L$,
$\mathcal{Z}(L)=\{x\in L\mid [L,x]=0\}$, is not zero. Fix  a nonzero element $u\in\mathcal{Z}(L)$.
The $k$-subspace $N=ku$ is an ideal of $L$ and $[L,N]=0$. Note that $R*U(N)$, the $k$-subalgebra of $R*U(L)$
generated by $R$ and $\bar{u}$, is a skew polynomial ring $R[\bar{u};\delta_u]$.

Let $x\in L$, then
\begin{equation}\label{eq:extensionderivation}
\bar{x}\bar{u}-\bar{u}\bar{x}=\overline{[x,u]}+t(x,u)=t(x,u)\in R\\
\end{equation}
Thus the restriction of the inner derivation of $R*U(L)$ determined
by $\bar{x}$ induces a $k$-derivation
$R[\bar{u};\delta_u]\rightarrow R[\bar{u};\delta_u]$, $f\mapsto
\bar{x}f-f\bar{x}$. Notice that it extends $\delta_x\colon
R\rightarrow R$, thus we will denote the extension again by
$\delta_x\colon R[\bar{u};\delta_u]\rightarrow R[\bar{u};\delta_u]$.

Introduce the new variable
$z=(\bar{u})^{-1}$, and let $R((z;\delta_{u}))$ be the power series
ring described in Section~\ref{sec:1.1}. Observe that if we want to extend $\delta_x$ to $R((z;\delta_u))$,
we have to define $\delta_x(z)=-z\delta_x(\bar{u})z$, and therefore
\begin{equation}\label{eq:imageofz}
 \delta_x(z)=\sum_{i=1}^\infty (-1)^i\delta_x^i(\bar{u})z^{i+1}.
\end{equation}

\begin{lem}\label{lem:extensionderivation}
For each $x\in L$, the derivation $\delta_x\colon R[\bar{u};\delta_u]\rightarrow R[\bar{u};\delta_u]$,
$f\mapsto \bar{x}f-f\bar{x}$ can be extended to a derivation
$$R((z;\delta_u))\rightarrow R((z;\delta_u)),\quad
\sum_{n}a_nz^n\mapsto \sum_n\delta_x(a_nz^n)=\sum_n\left(\delta_x(a_n)z^n+a_n\delta_x(z^n)\right),$$
where $\delta_x(z)$ is defined as in \eqref{eq:imageofz}.
\end{lem}

\begin{proof}
Set $S=R\langle z\ ;\ za=az-z\delta_u(a)z,\ a\in R   \rangle$. So
that, the $k$-algebra  $S$ is isomorphic to $R\coprod _k k[z]$
modulo the two-sided ideal generated by $\{za=az-z\delta_u(a)z,\
a\in R \}$. Let  $\varepsilon_1\colon k[z]\to S$ and let
$\varepsilon _2 \colon R\to S$ be the induced $k$-algebra
homomorphism. By the universal property of the coproduct, there is a
ring homomorphism $\varphi\colon S\to R[[z;\delta_u]]$ such that,
for any $n\ge 1$, $\varphi (\varepsilon _1(z^n))=z^n$ and, for any
$a\in R$, $\varphi (\varepsilon _2(a))=a$. Therefore $\varepsilon
_1$ and $\varepsilon _2$ are injective homomorphisms. To ease the
notation, we just identify $R$ and $k[z]$ with their image in $S$
without making any reference to the embeddings $\varepsilon _1$ and
$\varepsilon _2$.

The existence of $\varphi$ also implies that the  powers of $z$ are right and left $R$ independent in $S$.
In addition,  by \cite[Proposition~4.1]{bergmancoproducts}, $S$ is generated by the powers of $z$. Since, by the definition of $S$, $zR\subseteq Sz$, for any $n\ge 1$ the ideal $Sz^n$ is two-sided.

Fix $s\in S$ and $n\ge 0$,   then $s=a_0+a_1z+\cdots +a_nz^n+r_n$ where $a_i\in R$ for $i\in \{0,\dots ,n\}$ and $r_{n}\in Sz^{n+1}$. Let $\pi _n\colon R[[z;\delta _u]]\to R[[z;\delta _u]]/R[[z;\delta _u]]z^n$ denote the canonical projection. Since $\pi _n\circ \varphi (s)=a_0+a_1z+\cdots +a_nz^n$, $a_0, \dots ,a_n$ are uniquely determined. This implies that $\pi _n\circ \varphi$ induces an isomorphism $S/Sz^n\cong R[[z;\delta _u]]/R[[z;\delta _u]]z^n$. Since $\varprojlim R[[z;\delta _u]]/R[[z;\delta _u]]z^n\cong R[[z;\delta _u]]$ we conclude that the completion of $S$ with respect to the topology induced by the two-sided ideals $\{Sz^n\}_{n\ge 0}$ is isomorphic to  $R[[z;\delta _u]]$ and that the isomorphism $\hat \varphi \colon \hat S\to R[[z;\delta _u]]$ is defined by $\hat\varphi (s)=\sum _{i=0}^\infty a_iz^i$. Finally, since $\bigcap _{n\ge 1}Sz^n=\{0\}$, $S$ embeds in $R[[z;\delta _u]]$ and this embedding sends $z\mapsto z$, $a\mapsto a$ for all $a\in R$ and
$za\mapsto za=\sum_{i\geq 1}(-1)^{i-1}\delta_u^{i-1}(a)z^i$.

Now we are ready to prove that $\delta_x$  extends to $R((z;\delta_u))$.
As a first step, we claim that $\delta_x$ can be extended to $S$ by setting $$\delta_x(z)=-z\delta_x(\bar{u})z.$$
To prove the claim we must show  that there is a morphism of $k$-algebras $\Phi\colon S\rightarrow \mathbb{T}_2(S)$,
$f\mapsto \bigl( \begin{smallmatrix} f & \delta_x(f) \\ 0 & f  \end{smallmatrix}\bigr)$, where
$\mathbb{T}_2(S)$ is the ring of $2\times 2$ upper triangular matrices over $S$.

There is a morphism of $k$-algebras $\Phi _1\colon k[z]\to \mathbb{T}_2(S)$ given by $\Phi _1(p(z))= \bigl( \begin{smallmatrix} p(z) & \delta_x(p(z)) \\ 0 & p(z) \end{smallmatrix}\bigr)$ for any $p(z)\in k[z]$.  There is also a morphism of  $k$-algebras $\Phi _2\colon R\to \mathbb{T}_2(S)$ given by $\Phi _2(a)= \bigl( \begin{smallmatrix} a & \delta_x(a) \\ 0 & a \end{smallmatrix}\bigr)$ for any $a\in R$. By the universal property of the coproduct, there is a unique algebra homomorphism $\Phi _3\colon R\coprod k[z]\to \mathbb{T}_2(S)$ such that $\Phi _3(z)=\Phi _1(z)$ and such that, for any $a\in R$, $\Phi _3(a)=\Phi _2(a)$.

We show that, for any $a\in R$, $za-az+z\delta_u(a)z\in \mathrm{Ker}\, \Phi _3$. This is equivalent to the matrix equality
$$\left(\begin{smallmatrix} z & \delta_x(z) \\ 0 & z \end{smallmatrix} \right) \left(\begin{smallmatrix} a & \delta_x(a) \\ 0 & a  \end{smallmatrix} \right)=
\left(\begin{smallmatrix} a & \delta_x(a) \\ 0 & a  \end{smallmatrix} \right)\left(\begin{smallmatrix} z & \delta_x(z) \\ 0 & z   \end{smallmatrix} \right)  -
 \left(\begin{smallmatrix} z & \delta_x(z) \\ 0 & z \end{smallmatrix} \right)\left(\begin{smallmatrix} \delta_u(a) & \delta_x(\delta_u(a))  \\ 0 & \delta_u(a)  \end{smallmatrix}
\right)\left(\begin{smallmatrix} z & \delta_x(z) \\ 0 & z \end{smallmatrix} \right), $$
which  yields
$$\left(\begin{smallmatrix}  za &  z\delta_x(a) + \delta_x(z)a \\ 0 & za \end{smallmatrix} \right)=
\left(\begin{smallmatrix} az &  a\delta_x(z)+\delta_x(a)z \\ 0 &  az \end{smallmatrix} \right) -
\left(\begin{smallmatrix}  z\delta_u(a)z &   z\delta_u(a)\delta_x(z)+ z\delta_x(\delta_u(a))z + \delta_x(z)\delta_u(a)z  \\  0 & z\delta_u(a)z \end{smallmatrix} \right)$$

Hence $za-az+z\delta_u(a)z\in \mathrm{Ker}\, \Phi _3$ if and only if the  equality
$$z\delta_x(a) + \delta_x(z)a \stackrel{(*)}{=}a\delta_x(z)+\delta_x(a)z + z\delta_u(a)\delta_x(z)+ z\delta_x(\delta_u(a))z + \delta_x(z)\delta_u(a)z $$
holds.

After substituting  $\delta_x(z)$ by $-z\delta_x(\bar{u})z$, the right hand side of the equality $(*)$ equals to
$$-az\delta_x(\bar{u})z+\delta_x(a)z+z\delta_u(a)z\delta_x(\bar{u})z-z\delta_x(\delta_u(a))z + z\delta_x(\bar{u})z\delta_u(a)z.$$
Now, the left hand side of $(*)$ is
\begin{eqnarray*}
\scriptstyle{ z\delta_x(a)+\delta_x(z)a } & \scriptstyle{=} & \scriptstyle{\delta_x(a)z-z\delta_u(\delta_x(a))z-z\delta_x(\bar{u})za} \\
& \scriptstyle{=} & \scriptstyle{\delta_x(a)z-z\delta_u(\delta_x(a))z-z\delta_x(\bar{u})az+z\delta_x(\bar{u})z\delta_u(a)z} \\
& \scriptstyle{=} & \scriptstyle{\delta_x(a)z-z\delta_u(\delta_x(a))z - z[\delta_x(\bar{u}),a]z - za\delta_x(\bar{u})z + z\delta_x(\bar{u})z\delta_u(a)z} \\
& \scriptstyle{=} & \scriptstyle{\delta_x(a)z-z\delta_u(\delta_x(a))z - z[\delta_x(\bar{u}),a]z -az\delta_x(\bar{u})z +z\delta_u(a)z\delta_x(\bar{u})z + z\delta_x(\bar{u})z\delta_u(a)z.}
\end{eqnarray*}
After eliminating equal terms on both sides of $(*)$, we see that it holds if and only if
$$-z\delta_u(\delta_x(a))z-z[\delta_x(\bar{u}),a]z = -z\delta_x(\delta_u(a))z.$$
Equivalently, $$-z([\delta_x,\delta_u](a)-[\delta_x(\bar{u}),a]    )z=0.$$
This last equality holds because by \eqref{eq:extensionderivation}
and Lemma~\ref{lem:conditionsforcrossedproduct}(ii),
$$[\delta_x,\delta_u](a)=[t(x,u),a]=[\delta_x(\bar{u}),a],\ \textrm{ for all } a\in R.$$
Therefore, $\Phi_3$ induces the map $\Phi \colon S \to \mathbb{T}_2(S)$ which must be a morphism of $k$-algebras. This finishes the proof of the claim.

\medskip

The embedding $S\hookrightarrow R[[z;\delta_u]]$   induces a morphism $\mathbb{T}_2(S)\rightarrow
\mathbb{T}_2(R[[z;\delta_u]])$. The
completion of $\mathbb{T}_2(S)$ with respect to the ideals
$\{\mathbb{T}_2(S)Z^n\}_{n\geq1}$, where
$Z=\bigl(\begin{smallmatrix}  z & 0 \\ 0 &
z\end{smallmatrix}\bigr)$, is $\mathbb{T}_2(R)[[Z;\Delta_u]]$,
where the derivation $\Delta_u$ is given by
$\Delta_u\bigl(\begin{smallmatrix}  a & b \\ 0 &
c\end{smallmatrix}\bigr)= \bigl(\begin{smallmatrix}  \delta_u(a) &
\delta_u(b) \\ 0 & \delta_u(c)\end{smallmatrix}\bigr)$ for all
$(\begin{smallmatrix}  a & b \\ 0 & c\end{smallmatrix}\bigr)\in
\mathbb{T}_2(R)$. Recall that $\mathbb{T}_2(R)[[Z;\Delta_u]]$ is
canonically isomorphic to $\mathbb{T}_2(R[[z;\delta_u]])$. We will
use this identification in what follows.

Note that the morphism $$\varphi _S\colon S\stackrel{\Phi}\rightarrow
\mathbb{T}_2(S)\rightarrow \mathbb{T}_2(R[[z;\delta_u]]), \quad
f\mapsto \bigl(\begin{smallmatrix}  f & \delta_x(f) \\ 0 & f
\end{smallmatrix}\bigr) \mapsto \bigl(\begin{smallmatrix}  f &
\delta_x(f) \\ 0 & f \end{smallmatrix}\bigr)
$$
satisfies that $\varphi _S(Sz^n)\subseteq \mathbb{T}_2(R)Z^n$ and thus induces morphisms
$\varphi_n\colon \frac{S}{Sz^n}\rightarrow \frac{\mathbb{T}_2(R)}{Z^n}$ such that for all $n\geq m$ the
diagram
$$\xymatrix{ \frac{S}{Sz^n} \ar[r] ^{\varphi _n}\ar[d]  &
\frac{\mathbb{T}_2(R)}{\mathbb{T}_2(R)Z^n} \ar[d]\\
   \frac{S}{Sz^m} \ar[r]^{\varphi _m}       &
 \frac{\mathbb{T}_2(R)}{\mathbb{T}_2(R)Z^m}}$$
 is commutative.
Therefore there exists a morphism of $k$-algebras
$$R[[z;\delta_u]]\cong \hat S\rightarrow \mathbb{T}_2(R[[z;\delta_u]]),
\quad \sum_{i\geq 0} a_iz^i\mapsto \left(\begin{array}{cc}
\sum_{i\geq 0}a_iz^i & \sum_{i\geq 0}\delta_x(a_iz^i) \\ 0 &
\sum_{i\geq 0}a_iz^i\end{array}\right),$$ where $\delta_x$ is the
composition  $S\stackrel{\delta_x}{\rightarrow} S\hookrightarrow
R[[z;\delta_u]]$. Thus the derivation $\delta_x\colon S\rightarrow
S$ extends to $\delta_x\colon R[[z;\delta_u]]\rightarrow R[[z;\delta_u]]$ as
$$\delta_x\left(\sum_{i\geq 0} a_iz^i\right)=\sum_{i\geq 0} \delta_x(a_iz^i).$$
Since $R((z;\delta_u))$ is the left Ore localization of
$R[[z;\delta_u]]$ at the set $\{1,z,\dotsc,z^n,\dotsc\}$, $\delta_x$
also extends to a derivation  of $R((z;\delta_u))$ in a unique way
(cf. Lemma~\ref{lem:extensionofderivations}).

Since $\bar{u}=z^{-1}$, the
equality $za=az-z\delta_u(a)z$ implies that $\bar{u}a=a\bar{u}+\delta_u(a)$ for each $a\in R$; hence $R[\bar{u};\delta _u]\hookrightarrow R((z;\delta_u))$. Also, as  $\delta _x(z^{-1})=-z^{-1}\delta _x(z)z^{-1}$, $\delta _x(z^{-1}) =\delta _x(\bar{u})$. So that the derivation $\delta _x$ has the  properties claimed in the statement.\end{proof}

\begin{coro}\label{coro:constructingringofseries}
 There exists a crossed product structure $R((z;\delta_u))*U\left( \frac{L}{N}\right)$
such that $$R*U(L)=R[\bar{u};\delta_u]*U\left( \frac{L}{N}\right)\hookrightarrow R((z;\delta_u))*U\left(\frac{L}{N}\right).$$
\end{coro}

\begin{proof}
By the proof of
Lemma~\ref{lem:importantpropertycrossedproduct}, we know that, for
each $w\in L/N$, there exists $x\in L$ such that the $k$-derivation
$\delta_w\colon R[\bar{u};\delta_u]\rightarrow R[\bar{u};\delta_u]$ given by the definition of
$R[\bar{u};\delta_u]*U(L/N)$ coincides with
$\delta_x\colon R[\bar{u};\delta_u]\rightarrow R[\bar{u};\delta_u]$. We  extend it to a $k$-derivation of
$R((z;\delta_{u}))$ as in Lemma~\ref{lem:extensionderivation},
we denote the extension also by $\delta_w$.
This gives a map $\xi \colon L/N\to \mathrm{Der} _k(R((z;\delta_{u})))$
which is $k$-linear by the proof of Lemma~\ref{lem:importantpropertycrossedproduct}.
By Remarks~\ref{rem:extensioncrossedproduct}, to obtain our result
it only remains to prove that condition (ii) in
Lemma~\ref{lem:conditionsforcrossedproduct} is satisfied.

By \eqref{eq:extensionderivation} and  Lemma~\ref{lem:extensionderivation},
\begin{equation} \label{eq:uniquenessofextension}
\delta_w(f)=\sum_{i\geq l }a_i'z^i \quad \textrm{ for any }
f=\sum_{i\geq l} a_iz^i\in R((z;\delta_{u})).
\end{equation}
That is, if the coefficients of degree smaller than $l$ of $f$ are
zero, then the coefficients of degree smaller than $l$ of
$\delta_w(f)$ are zero.

 Let $f=\sum_{i\geq l} a_iz^i \in
R[[z;\delta_{u}]]$ and $w_1,w_2\in L/N$, we want to prove that
$$[\delta_{w_1},\delta_{w_2}](f)=\delta_{[w_1,w_2]}(f)-[f,t(w_1,w_2)].$$
Fix $p>l$ and set  $f=f_1+ f_2$ where $f_1=\sum_{i\geq l}^pa_iz^i$
and $f_2=\sum_{i\geq p+1}a_i z^i$. The derivations $[\delta_{w_1},
\delta_{w_2}]$ and $\delta_{[w_1,w_2]}-\partial_{t(w_1,w_2)}$
coincide on $R[\bar{u};\delta_u]$, hence on $z$ by
Lemma~\ref{lem:uniqueextensionderivation}, and therefore on $f_1$.
Thus
$$[\delta_{w_1}, \delta_{w_2}](f)=[\delta_{w_1}, \delta_{w_2}](f_1)+[\delta_{w_1}, \delta_{w_2}](f_2)=
\delta_{[w_1,w_2]}(f_1)-\partial_{t(w_1,w_2)}(f_1) + \sum_{i\geq
p+1}c_iz^i.$$ Since $p$ was arbitrary, both derivations coincide on
$f$ and  we obtain our result.
\end{proof}

\subsection{The case of hypercentral Lie
algebras}\label{subsec:nilpotentcase}

A Lie $k$-algebra $L$ is \emph{hypercentral} if there exist an
ordinal $\nu$ and a chain of ideals $\{L_\mu\}_{\mu\leq\nu}$ of $L$
that satisfy the following conditions:
\begin{enumerate}[(i)]
\item $L_0=0,\ L_\nu=L$.
\item $L_\mu\subset L_{\mu+1}$ for all $0\leq \mu<\nu$.
\item $L_{\mu'}=\bigcup_{\mu<\mu'}L_\mu$ for all limit ordinals $\mu'\leq
\nu$.
\item $[L,L_{\mu+1}]\subseteq L_\mu$ for all $\mu<\nu$, or equivalently, $L_{\mu+1}/L_\mu$ is contained in the center of $L/L_{\mu}$.
\end{enumerate}
We will say that $\{L_\mu\}_{\mu\leq\nu}$ is an \emph{hypercentral
series} of $L$.

For our purposes, the most important example of hypercentral Lie
algebra  is that of a nilpotent Lie algebra. Indeed, if $L$ is a
nilpotent Lie $k$-algebra,  it is enough to choose $L_0=0$,
$L_1=\mathcal{Z}(L)$, and for $i\geq 1$,
$L_{i+1}/L_i=\mathcal{Z}(L/L_i)$. It is not difficult to prove
that any hypercentral Lie algebra is locally nilpotent.

\medskip

Fix a hypercentral Lie $k$-algebra $L$ together with a
hypercentral series $\{L_\mu\}_{\mu\leq\nu}$ of $L$. Let $R$ be a
 $k$-algebra  and consider a crossed product
$R*U(L)$.

For each $0\leq \mu<\nu$, we pick in $L_{\mu+1}$ a set of elements
$\mathcal{B}_\mu$ which gives a basis of $L_{\mu+1}/L_{\mu}$, and we
endow $\mathcal{B}_\mu$ with a well-ordered set structure.  Set
$\mathcal{B}=\bigcup_{\mu<\nu} \mathcal{B}_\mu$. Observe that
$\mathcal{B}$ is a basis of $L$. Then we order $\mathcal{B}$
extending the ordering in each $\mathcal{B}_\mu$ in the following
way: given $u_1\in \mathcal{B}_{\mu_1}$ and $u_2\in
\mathcal{B}_{\mu_2}$ we set
\begin{equation}\label{eq:orderingbasis}
u_1<u_2\quad \textrm{iff}\quad  \left\{\begin{array}{l} \mu_1<\mu_2,
\textrm{ or }
\\ \mu_1=\mu_2  \textrm{ and  } u_1 \textrm{ is smaller than } u_2
 \textrm{ in  } \mathcal{B}_{\mu_1}. \end{array}\right.
\end{equation}
Then $(\mathcal{B},<)$ is a well-ordered set. Thus, we can suppose
that there exists an ordinal $\varepsilon$ such that $\mathcal{B}=\{u_\gamma\}_{0\leq \gamma<\varepsilon}$ and $u_{\gamma _1}\le u_{\gamma _2}$ if and only if $\gamma _1\le \gamma_2$.

For each $0\leq\beta\leq\varepsilon$, set $N_\beta$ as the
$k$-subspace of $L$ generated by $\{u_\gamma\mid\gamma<\beta\}$. By
convention, $N_0=0$. Observe that $N_\beta$ is an ideal of $L$,
hence a Lie subalgebra of $L$, and that $$[L,u_\beta]\subset
N_\beta.$$

By transfinite induction, we construct a  ring of series
$R((N_\beta))$  and  a crossed product $R((N_\beta))*U(L/N_\beta)$,
for each $\beta\leq \varepsilon$, such that the following properties
are satisfied for $\gamma<\beta\leq\varepsilon$

\begin{enumerate}[(a)]
\item $R((N_\gamma))\hookrightarrow R((N_\beta))$,
\item $R*U(L)\hookrightarrow
R((N_\gamma))*U(L/N_\gamma)\hookrightarrow
R((N_\beta))*U(L/N_\beta)$ extending the embedding of (a) in the
natural way.
\end{enumerate}

We define $R((N_0))=R$. Let $0<\beta$ be an ordinal and suppose that
we have defined $R((N_\gamma))$ for all $\gamma<\beta$ such that
conditions (a) and (b) are satisfied. Suppose first that $\beta$ is
not a limit ordinal, thus $\beta=\gamma+1$ for some ordinal
$\gamma$. Set $R_\gamma=R((N_\gamma))$ and
$T_\gamma=R_\gamma[\bar{u}_\gamma;\delta_{u_\gamma}]$. Introduce a
new variable $z_\gamma=(\bar{u}_\gamma)^{-1}$. Define
$R((N_{\beta}))=R_\gamma ((z_\gamma;\delta_{u_\gamma}))$. By
Corollary~\ref{coro:constructingringofseries}, there exists a
crossed product structure $R((N_\beta))*U(L/N_{\beta})$ such that
$$R*U(L)\hookrightarrow R((N_\gamma))*U(L/N_\gamma)
\hookrightarrow R((N_{\beta}))*U(L/N_{\beta}).$$ Thus conditions (a)
and (b) are satisfied.

Suppose now that $\beta$ is a limit ordinal.  Define
$R((N_\beta))=\bigcup_{\gamma<\beta}R((N_\gamma))$. Set
$$M=\varinjlim_{\gamma<\beta} R((N_\gamma))*U(L/N_\gamma).$$
We want to prove that $M$ has a natural crossed product structure of
$R((N_\beta))$ by $U(L/N_\beta)$. We show that $M$ satisfies
conditions (i) and (ii) in the definition of a crossed product. For
that it is helpful to have in mind the proof of
Lemma~\ref{lem:importantpropertycrossedproduct}.

Consider
$\mathcal{S}=\{u_\alpha\}_{\beta\leq\alpha<\varepsilon}\subseteq L$.
Let $\mathcal{M}$ be the set of standard monomials on $\mathcal{S}$.
Abusing notation, we may suppose that $\mathcal{M}\subseteq
R((N_\gamma))*U(L/N_\gamma)$ for all $\gamma<\beta$, and  the
embedding $R((N_{\gamma_1}))*U(L/N_{\gamma_1})\hookrightarrow
R((N_{\gamma_2}))*U(L/N_{\gamma_2})$ can be seen as the identity on
$\mathcal{M}$ for $\gamma_1<\gamma_2<\beta$.

Let $f\in M$. There exists $\gamma<\beta$ such that $f$ is a finite
sum of the form $\sum_{m\in\mathcal{M}}f_mm$ with $f_m\in
R((N_\gamma))$ for all $m\in\mathcal{M}$. Moreover, given
$f_{m_1},\dotsc, f_{m_n}\in R((N_\beta))$, there exists
$\gamma<\beta$ such that $f_{m_1},\dotsc,f_{m_n}\in R((N_\gamma))$,
and $\sum_{i=1}^n f_{m_i}m_i=0$ implies that
$f_{m_1}=\dotsb=f_{m_n}=0$. Hence $M$ has the additive structure of
$R((N_\beta))\otimes_k U(L/N_\beta)$

For each $\gamma<\beta$, we identify the subspace of $L$ generated
by $\mathcal{S}$ with a subspace of $L/N_\gamma$ in the natural way.
Let $x\in L$ be any $k$-linear combination of elements in
$\mathcal{S}$. For each $\gamma<\beta$, the crossed product
structure $R((N_\gamma))*U(L/N_\gamma)$ defines a derivation
$\delta_{x,\gamma}\colon R((N_\gamma))\rightarrow R((N_\gamma))$.
Moreover, if $\gamma_1<\gamma_2<\beta$, since
$R((N_{\gamma_1}))\subseteq R((N_{\gamma_2}))$ and
$R((N_{\gamma_1})*U(L/N_{\gamma_1}))\hookrightarrow
R((N_{\gamma_2}))*U(L/N_{\gamma_2})$ we have that
$\delta_{x,\gamma_1}$ equals $\delta_{x,\gamma_2}$ on
$R((N_{\gamma_1}))$. Set $\delta_x\colon R((N_\beta))\rightarrow
R((N_\beta))$ in the natural way, that is, for each $f\in
R((N_\beta))$ there exists $\gamma<\beta$ such that $f\in
R((N_\gamma))$, then we set $\delta_x(f)=\delta_{x,\gamma}(f)$. Then
$\delta\colon L/N_\beta\rightarrow \Der_k(R((N_\beta)))$, $x\mapsto
\delta_x$, is defined, where we are identifying the subspace of $L$
generated by $\mathcal{S}$ and $L/N_\beta$ in the natural way. Let
$x,y\in L/N_\beta$ and $f\in R((N_\beta))$. The equality
$\bar{x}f=f\bar{x}+\delta_x(f)$ holds because $f\in R((N_\gamma))$
for some $\gamma<\beta$. Let $t\colon L\times L\rightarrow
R((N_\beta))$ be given by the crossed product structure of
$(R*U(N_\beta))*U(L/N_\beta)\cong R*U(L)$. Hence, in particular
$t(x,y)\in R*U(N_\beta).$ Then
$\bar{x}\bar{y}-\bar{y}\bar{x}=\overline{[x,y]}+t(x,y)$. Thus
conditions \eqref{def:crossedproduct1} and
\eqref{def:crossedproduct2} are satisfied. Therefore
$M=R((N_\beta))*U(L/N_\beta)$ and $R*U(L)\hookrightarrow
R((N_\gamma))*U(L/N_\gamma)\hookrightarrow
R((N_\beta))*U(L/N_\beta)$ for $\gamma<\beta$.

We then define $R((L))=R((N_\varepsilon))$.

\begin{rems}\label{rem:constructionringofseries}
\begin{enumerate}[(a)]
\item The ring $R((L))$ depends on the order $<$   in
\eqref{eq:orderingbasis} of the basis
$\{u_\gamma\}_{0\leq\gamma<\varepsilon}$ of $L$ obtained from the
hypercentral series $\{L_\mu\}_{\mu\leq\nu}$ of $L$. The same
hypercentral series $\{L_\mu\}_{\mu\leq\nu}$ can give rise to
different rings of series $R((L))$ because $R((L))$ depends on the
basis $\mathcal{B}_\mu$ and the different well-ordered set
structures that each $\mathcal{B}_\mu$ can be given. Also, different
hypercentral series can give rise to the same ring of series
$R((L))$ if we choose the same basis $\{u_\gamma\}_{0\leq
\gamma<\varepsilon}$ and the same order $<$ obtained as in
\eqref{eq:orderingbasis}.

\item By construction, if $\beta<\varepsilon$, then
\begin{equation}\label{eq:ringofseries}
R((L))=R((N_\beta))((L/N_\beta))
\end{equation}
where we are identifying the ordered set $\{u_\alpha\}_{\beta\leq
\alpha<\varepsilon}$ with an ordered basis of the hypercentral Lie
algebra $L/N_\beta$  in the natural way.

\item If $R$ is a domain then $R((L))$ is also a domain.

\item If $L'$ is a subalgebra of $L$ with a basis $\mathcal{B}'\subseteq
\mathcal{B}$, where we understand that the order of $\mathcal{B}'$
is inherited from the one in $\mathcal{B}$, then
$R((L'))\hookrightarrow R((L))$ in the natural way. Indeed, if we
define $N_{\beta}'=\{u_\gamma\mid u_\gamma\in \mathcal{B}',\
\gamma<\beta\}$, then $R((N_{\beta}'))\subseteq R((N_{\beta}))$ in
the natural way for each $0\leq\beta<\varepsilon$. \qed
\end{enumerate}
\end{rems}

Now we want to define the so called \emph{least element map}
$\ell\colon R((L))\rightarrow R$. Let $f\in R((L))$. Let $\beta_1$
be the least ordinal such that $f\in R((N_{\beta_1}))$.
Note that $\beta_1$ is not a limit ordinal. If $\beta_1=0$, i.e.
$f\in R$, we define $\ell(f)=f$. Suppose $\beta_1\neq 0$. Thus there
exists an ordinal $\gamma_1$ such that $\beta_1=\gamma_1+1$. By
construction,
$R((N_{\beta_1}))=R((N_{\gamma_1}))((z_{\gamma_1};\delta_{u_{\gamma_1}}))$.
Hence $f$ is a series in $z_{\gamma_1}$ with coefficients in
$R((N_{\gamma_1}))$. Let $f_1\in R((N_{\gamma_1}))$ be the
coefficient of the least element in $\supp f$ as a series in
$z_{\gamma_1}$. Let $\beta_2$ be the least ordinal such that $f_1\in
R((N_{\beta_2}))$. If $\beta_2=0$, we define $\ell(f)=f_1$. If
$\beta_2\neq0$, there exists an ordinal $\gamma_2$ such that
$\beta_2=\gamma_2+1$. By construction,
$R((N_{\beta_2}))=R((N_{\gamma_2}))((z_{\gamma_2};\delta_{u_{\gamma_2}}))$.
Thus $f_1$ is a series in $z_{\gamma_2}$ with coefficients in
$R((N_{\gamma_2}))$. Let $f_2\in R((N_{\gamma_2}))$ be the
coefficient of the least element in $\supp f_1$ as a series in
$z_{\gamma_2}$. Let $\beta_3$ be the least ordinal such that $f_2\in
R((N_{\beta_3}))$. If $\beta_3=0$, we define $\ell(f)=f_2$. If
$\beta_2\neq0$, there exists an ordinal $\gamma_3$, \dots

Continuing in this way we obtain a descending chain of nonlimit
ordinals
$$\beta_1=\gamma_1+1>\beta_2=\gamma_2+1>\beta_3=\gamma_3+1>\dotsb$$
Note that if $\beta_r\neq0$, then $\beta_r=\gamma_r+1$ and $f_r\in
R((N_{\gamma_r}))((z_{\gamma_r};\delta_{u_{\gamma_r}}))$ and
$\beta_{r+1}$ is defined. Hence, since the set of ordinals
$\{\beta\mid 0\leq\beta<\varepsilon\}$ is a well ordered set, there
exists a natural $n$ such that $\beta_n=0$. We define
$\ell(f)=f_{n-1}$. We say that $\ell(f)$ is the \emph{least element}
of $f$.

We collect some properties of the least element map in the following
Lemma.

\begin{lem}\label{lem:propertiesleastelementmap}
Let $\ell\colon R((L))\rightarrow R$ be the least element map. Let
$f,g\in R((L))$. The following hold true:
\begin{enumerate}[\rm(i)]
\item $\ell(f)=f$ if, and only if, $f\in R$.
\item $\ell(f)=0$ if, and only if, $f=0$.
\item Let $L'$ be a subalgebra of $L$ with a basis
$\mathcal{B}'\subseteq \mathcal{B}$, where we understand that the
order of $\mathcal{B}'$ is inherited from the one in $\mathcal{B}$.
If $f\in R((L'))$, then the least element of $f$ viewed as an
element of $R((L'))$ coincides with $\ell(f)$.

\item  If $R$ is a domain, then $\ell(fg)=\ell(f)\ell(g)$.
\item If $\ell(f)$ is invertible in $R$, then $f$ is invertible in
$R((L))$. If $R$ is a domain, the converse is true.
\end{enumerate}
\end{lem}
\begin{proof}
(i) and (ii) follow easily from the construction.

(iii) follows by construction,  defining $N_{\beta}'$ as in
Remarks~\ref{rem:constructionringofseries}(d) and by induction on
$\beta$.

We prove (iv) by induction on the least ordinal $\beta$ such that
$f,g\in R((N_\beta))$. Observe that $\beta$ is not a limit ordinal.
If $\beta=0$, the result is clear by (i). Suppose that $\beta>0$ and
the result is true for $\gamma<\beta$. As $\beta=\gamma+1$,
 $f,g\in
R((N_\beta))=R((N_\gamma))((z_\gamma;\delta_{u_\gamma}))$. If both
$f,g\in R((N_\gamma))((z_\gamma;\delta_{u_\gamma}))\setminus
R((N_\gamma))$. Then $f_1g_1=(fg)_1\in R((N_\gamma))$ because of the
way  series in one indeterminate are multiplied and the fact that
$R((N_\gamma))$ is a domain (since $R$ is). Now observe that, by
construction, $\ell(f_1)=\ell(f)$, $\ell(g_1)=\ell(g)$ and
$\ell(fg)=\ell((fg)_1)$. Thus applying the induction hypothesis
$$\ell(fg)=\ell((fg)_1)=\ell(f_1g_1)=\ell(f_1)\ell(g_1)=\ell(f)\ell(g).$$
If $f\in R((N_\gamma))((z_\gamma;\delta_{u_\gamma}))\setminus
R((N_\gamma))$ but $g\in R((N_\gamma))$, then $(fg)_1=f_1g\in
R((N_\gamma))$. Using the induction hypothesis,
$$\ell(fg)=\ell((fg)_1)=\ell(f_1g)=\ell(f_1)\ell(g)=\ell(f)\ell(g).$$
The remaining case is done analogously.

(v) Suppose that $\ell(f)$ is invertible. Set $f_0=f$. By definition
of $\ell(f)$, there exists a descending chain of nonlimit ordinals
$$\beta_1=\gamma_1+1>\beta_2=\gamma_2+1>\dotsb >\beta_{n-1}=\gamma_{n-1}+1>\beta_n=0$$
such that $f_{i-1}\in
R((N_{\gamma_i}))((z_{\gamma_i};\delta_{u_{\gamma_i}}))$ and $f_i$
is the coefficient of the least element in $\supp f_{i-1}$ as a
series in $z_{\gamma_i}$, and $\ell(f)=f_{n-1}$. The fact that
$f_{n-1}$ is invertible in $R\subset R((N_{\gamma_{n-1}}))$ implies
that $f_{n-2}$ is invertible in
$R((N_{\gamma_{n-1}}))((z_{\gamma_{n-1}};\delta_{u_{\gamma_{n-1}}}))\subset
R((N_{\gamma_{n-2}}))$. Hence $f_{n-3}$ is invertible in
$R((N_{\gamma_{n-2}}))((z_{\gamma_{n-2}};\delta_{u_{\gamma_{n-2}}}))$
\dots  Continuing in this way, we get that $f_1\in
R((N_{\gamma_2}))((z_{\gamma_2};\delta_{u_{\gamma_2}}))\subseteq
R((N_{\gamma_1}))$ is invertible, and therefore $f=f_0 \in
R((N_{\gamma_1}))((z_{\gamma_1};\delta_{u_{\gamma_1}}))$ is
invertible.

Suppose now that $R$ is a domain and that  $f$ is invertible.
Applying (i) and (iv), we get
$\ell(f^{-1})\ell(f)=\ell(f^{-1}f)=1=\ell(1)=
\ell(ff^{-1})=\ell(f)\ell(f^{-1})$.
\end{proof}

As a first outcome, we obtain a slight generalization of
\cite[Section~5]{Lichtmanuniversalfields}.

\begin{coro}\label{coro:fieldoffractionsnilpotentlie}
Let $L$ be an hypercentral Lie $k$-algebra.  Let $K$ be a field with
$k$ as a central subfield. Any crossed product $K*U(L)$ is an Ore
domain and $K((L))$ is a field that contains the Ore field of
fractions of $K*U(L)$.
\end{coro}

\begin{proof}
Any hypercentral Lie $k$-algebra is locally nilpotent. Thus $K*U(L)$ is locally an iterated skew
polynomial ring $K[x_1;\delta_1]\dotsb[x_n;\delta_n]$, which is an Ore domain. We have already seen that
$K*U(L)\hookrightarrow K((L))$. Now $K((L))$ is a field by Lemma~\ref{lem:propertiesleastelementmap}(v). By the universal property of the Ore localization, the Ore field of fractions of $K*U(L)$ is contained in $K((L))$.
\end{proof}

\subsection{The residually nilpotent
case}\label{subsec:residuallynilpotentcase}

Let $H$ be a Lie $k$-algebra. We say that $H$ is \emph{residually
nilpotent} if $H$ has a descending sequence of ideals
\begin{equation}\label{eq:residuallynilpotent}
 H=H_1\supseteq H_2\supseteq \dotsb\supseteq H_i\supseteq H_{i+1} \supseteq \dotsb
\end{equation}
with $[H,H_i]\subseteq H_{i+1}$ for all $i$, and such that
$\bigcap_{i\geq 1}^\infty H_i=0$. In this event, we call
$\{H_i\}_{i\geq 1}$ an \emph{RN-series} of $H$. The RN-series
$\{H_i\}_{i\geq 1}$ satisfies the \emph{Q-condition} if, for each
$i$, there exists a set of elements $\mathcal{C}_i$ of $H_i$ which
gives a basis of $H_i/H_{i+1}$ such that
$\mathcal{C}=\bigcup_{i=1}^\infty \mathcal{C}_i$ is a basis of $H$.
We also say that $\mathcal{C}$ is a \emph{Q-basis} of $H$.

Given a Q-basis $\mathcal{C}$, a \emph{canonical ordering} of
$\mathcal{C}$ is an ordering $<$ of $\mathcal{C}$ obtained as we are
going to see next. First we give an (arbitrary) well ordered set
structure to $\mathcal{C}_i$ for each $i\geq 1$.  Then we order
$\mathcal{C}$ extending the order in each $\mathcal{C}_i$ in the
following way: given $u_1\in\mathcal{C}_{i_1}$ and
$u_2\in\mathcal{C}_{i_2}$ we set
\begin{equation}\label{eq:orderingbasis2}
u_1<u_2\quad \textrm{iff}\quad  \left\{\begin{array}{l} i_1>i_2,
\textrm{ or }
\\ i_1=i_2  \textrm{ and  } u_{i_1} \textrm{ is smaller than }
u_{i_2}
 \textrm{ in  } \mathcal{C}_{i_1}. \end{array}\right.
\end{equation}
Notice that there may exist infinite canonical orderings of
$\mathcal{C}$.

We remark that $(\mathcal{C},<)$ need not be a well-ordered set, but
$\bigcup_{i=1}^m\mathcal{C}_i$ can be seen as a well ordered basis
of $H/H_{m+1}$ for any $m$ under the obvious identification.

Not all residually nilpotent Lie $k$-algebras have a
$Q$-basis. Important examples of residually nilpotent Lie
$k$-algebras with Q-basis are the following.

\begin{exs} \label{ex:qbasis}
\begin{enumerate}[(1)]
\item Suppose that $H$ is a nilpotent Lie $k$-algebra. Let  $H=H_1\supseteq \dotsb \supseteq
H_{n+1}=0$ be an RN-series. If $\mathcal{C}_i$ is a set of elements
of $H_i$ which gives a basis of $H_i/H_{i+1}$, then clearly
$\mathcal{C}=\bigcup_{i=1}^n\mathcal{C}_i$ is a Q-basis of $H$.

\item Suppose that $H$ is a \emph{graded} Lie $k$-algebra, that is,
there exists a sequence $\{N_i\}_{i\geq 1}$ of subspaces of $H$ such
that $H=\bigoplus_{i=1}^\infty N_i$ and $[N_i,N_j]\subseteq N_{i+j}$
for all $i,j\geq 1$. If we now define $H_i=\bigoplus_{j\geq i} N_j$,
and $\mathcal{C}_i$ as any basis of $N_i$ for each $i\geq 1$, then
it is easy to see that $\{H_i\}_{i=1}^\infty$ is an RN-series and
$\mathcal{C}=\bigcup_{i=1}^\infty\mathcal{C}_i$ a Q-basis of $H$.

Examples of these algebras are the Lie algebras arising
from torsion-free nilpotent and residually torsion-free nilpotent
groups using the lower central series (of the groups),  and the
graded Lie algebras that appear in
\cite[Examples~A,B,C,D]{ShalevZelmanov}. \qed
\end{enumerate}
\end{exs}

Fix a residually nilpotent Lie $k$-algebra $H$ with an RN-series
$\{H_i\}_{i=1}^\infty$ that has a Q-basis
$\mathcal{C}=\bigcup_{i=1}^\infty\mathcal{C}_i$ and a canonical
ordering of $\mathcal{C}$.

Note that for each $n>m\geq 1$, $H/H_m$ and $H_m/H_n$ are nilpotent {and hence hypercentral}
Lie $k$-algebras. Moreover $$H_m/H_m=0<H_{m-1}/H_m<\dotsb <H_2/H_m<H/H_m$$ is a chain of ideals of
$H/H_m$ with $[H/H_m,H_p/H_m]\subseteq H_{p+1}/H_m$, and $$H_n/H_n=0< H_{n-1}/H_n<\dotsb< H_m/H_n$$
is a chain of ideals of $H_m/H_n$ with $[H_m/H_n,H_{m+p}/H_n]\subseteq H_{m+p+1}/H_m$.

For $1\leq i\leq m-1$, let $\mathcal{B}_{m,i}$ be the basis of
$\frac{H_i/H_m}{H_{i+1}/H_m}\cong H_i/H_{i+1}$ obtained via the natural identification with $\mathcal{C}_i$.
Set $\mathcal{B}_m=\bigcup_{i=1}^{m-1}\mathcal{B}_{m,i}$ be the basis of $H/H_m$ with the well order
inherited from $\bigcup_{i=1}^{m-1}\mathcal{C}_i$.

Let $R$ be a $k$-algebra and consider a crossed product $R*U(H)$.

For each $m\geq 1$, set $R_m=R*U(H_m)$.  Then, with each basis
$\mathcal{B}_m$ fixed, we can construct the embedding
$R_m*U(H/H_m)\hookrightarrow R_m((H/H_m))$. If $n>m$, since
$R_m=R_n*U(H_m/H_n)$ and $R_n((H/H_n))=R_n((H_m/H_n))((H/H_m))$, we
obtain the commutativity of the following diagram

\begin{equation}\label{eq:commutativediagramseries1}
 \xymatrix{ R_m*U(H/H_m)\ar@{^{(}->}[r] \ar@{=}[d]  &
R_m((H/H_m))=R_m*U(H_m/H_n)((H/H_m))\ar@{^{(}->}[d]\\
R_n*U(H/H_n) \ar@{^{(}->}[r] &
R_n((H/H_n))=R_n((H_m/H_n))((H/H_m))}
\end{equation}

It allows us to define $$R((H))=\varinjlim_m R_m((H/H_m)).$$

For each $n>m\geq 1$, let $\ell_m\colon R_m((H/H_m))\rightarrow R_m$
be the least element map, and let $t_m\colon R_m\rightarrow R_{m+1}$
be the least element map of $R_{m+1}((H_m/H_{m+1}))$ restricted to
$R_m$ (or equivalently, the restriction of $\ell_{m+1}$ to $R_m$ by
Lemma~\ref{lem:propertiesleastelementmap}(iii)). The commutativity
of the diagram
\begin{equation}\label{eq:commutativediagramseries2}
 \xymatrix{ R_m((H/H_m))\ar[r]^{\ell_m} \ar@{^{(}->}[d]  &
R_m\ar[d]^{t_m}\\
R_{m+1}((H/H_{m+1}))\ar[r]^{\qquad \ell_{m+1}} &
R_{m+1}}
\end{equation}
follows from \eqref{eq:commutativediagramseries1}.

Note that, because of Lemma~\ref{lem:propertiesleastelementmap}(i),
each $t_m$ is the identity on $R_{m+1}\subseteq R_m$, and hence on
$R$.

We claim that if $f\in R*U(L)$, there exists $m\geq  1$ such that
$\ell_m(f)\in R$. Indeed, we may express $f=\sum_{i=1}^n a_im_i$
where each $a_i\in R$ and each $m_i$ is a standard monomial in the
set $\mathcal{C}$. Thus there exists $m\geq 1$ such that $f$ is an
$R$-linear combination of the standard monomials in
$\bigcup_{i=1}^{m-1}\mathcal{C}_i$. Now, by definition of
$\ell_m\colon R_m((H/H_m))\rightarrow R_m$, it follows that
$\ell_m(f)\in R$, and the claim is proved.

Let now $f\in R((H))$. There exists $m\geq 1$ such that $f\in
R_m((H/H_m))$. By the claim and the commutativity of
\eqref{eq:commutativediagramseries2}, there exists $m_0$ such that
$\ell_{m_0}(f)\in R$. The commutativity of
\eqref{eq:commutativediagramseries2} and the fact that $\ell_l$  is
the identity on $R$ for each $l$   implies that
$\ell_n(f)=\ell_{m_0}(f)$ for all $n\geq m_0$. Thus we have a well
defined map $\ell\colon R((H))\rightarrow R$ where for each $f\in
R((H))$, $\ell(f)=\ell_{m_0}(f)$ where $m_0$ is any natural such
that $\ell_{m_0}(f)\in R$. The map $\ell\colon R((H))\rightarrow R$
is called the \emph{least element map} of $R((H))$, and $\ell(f)$
the \emph{least element} of $f\in R((H))$.

\begin{lem}\label{lem:propertiesleastelementmapresidually}
The least element map $\ell\colon R((H))\rightarrow R$ satisfies the
properties (i)-(v) in Lemma~\ref{lem:propertiesleastelementmap}.
\end{lem}

\begin{proof}
(i) and (ii) are clear from the construction.

(iii) Define $H_{m}'=H_m\cap H'$ and $R_{m}'=R*U(H_{m}')$. Then
$R((H'))=\varinjlim_m R_{m}'((H'/H_{m}'))$ and the result follows
from Lemma~\ref{lem:propertiesleastelementmap}(iii).

(iv) Let $f,g\in R((H))$. There exists $m\geq 1$ such that $f,g\in
R_m((H/H_m))$ and $\ell_m(f), \ell_m(g)\in R$. Now apply
Lemma~\ref{lem:propertiesleastelementmap}(iv).

(v) if $\ell(f)$ is invertible, then $\ell_m(f)$ is invertible for
some $m$ such that $f\in R_m((H/H_m))$. By
Lemma~\ref{lem:propertiesleastelementmap}(v), $f$ is invertible in
$R_m((H/H_m))$, and therefore in $R((H))$.  If $R$ is a domain, then
so is $R_m$ for all $m$. Now apply
Lemma~\ref{lem:propertiesleastelementmap}(v).
\end{proof}

From all this, we obtain the extension of \cite[Theorem~2]{Lichtmanuniversalfields}
to crossed products $K*U(H)$. More precisely, it follows from
Lemma~\ref{lem:propertiesleastelementmapresidually}(v),

\begin{coro}
Let $H$ be a residually nilpotent Lie $k$-algebra $H$ with an
RN-series $\{H_i\}_{i=1}^\infty$ that has a Q-basis
$\mathcal{C}=\bigcup_{i=1}^\infty\mathcal{C}_i$. Let $K$ be a field
with $k$ as a central subfield. For any crossed product $K*U(H)$ and
canonical ordering, the ring of series $K((H))$ is a field that
contains $K*U(H)$. \qed
\end{coro}

The subfield of $K((H))$ generated by $K*U(H)$ will be denoted by
$K(H)$.

\subsection{Main results}

The next result gives a condition that ensures when two fields of
fractions of a crossed product are isomorphic. It is the
generalization of \cite[Section~6,
Corollary]{Lichtmanuniversalfields} to crossed products. Although
weaker, it should be seen as a similar result to
\cite[Theorem]{Hughes}.

\begin{theo}\label{theo:Hughestheoremforliealgebras}
 Let $H$ be a
residually nilpotent Lie $k$-algebra with an RN-series
$\{H_i\}_{i=1}^\infty$ that has a Q-basis
$\mathcal{C}=\bigcup_{i=1}^\infty \mathcal{C}_i$. Let $K$ be a field
with $k$ as a central subfield. Consider a crossed product $K*U(H)$
and suppose that it has a field of fractions $K*U(H)\hookrightarrow
D$. For each $m\geq 1$, denote by $D_m$ the subfield of $D$
generated by $K*U(H_m)$. Assume that, for each $m\geq 1$, the
standard monomials in  $\bigcup_{i=1}^{m-1} \mathcal{C}_i$ are
linearly independent over $D_m$. Then $K(H)$ and $D$ are isomorphic
fields of fractions of $K*U(H)$.
\end{theo}

\begin{proof}
For each $m\geq 1$, set $R_m=K*U(H_m)$ and consider $K*U(H)$ as
$R_m*U(H/H_m)$. Fix $x\in H/H_m$, the derivation $\delta_x$ of $R_m$
extends to $D$ as the inner derivation $\delta
_x(d)=\bar{x}d-d\bar{x}$;  since $\delta _x(R_m)\subseteq R_m$ and
$D_m$ is a field of fractions of $R_m$, we deduce from
Lemma~\ref{lem:uniqueextensionderivation} that $\delta
_x(D_m)\subseteq D_m$. Therefore, the subring of $D$ generated by
$D_m$ and $\bigcup_{i=1}^{m-1} \mathcal{C}_i$ is a crossed product
$D_m*U(H/H_m)$ because conditions (i) and (ii) in the definition of
crossed product are easily verified.

For each $m\geq 1$, we can consider $D_m*U(H/H_m)\hookrightarrow
D_m((H/H_m))$ because $H/H_m$ is a nilpotent Lie $k$-algebra. Since
$R_m*U(H/H_m)\hookrightarrow D_m*U(H/H_m)\hookrightarrow D$,   $D$
is a field of fractions of $R_m*U(H/H_m)=K*U(H)$ and $D_m*U(H/H_m)$
is an Ore domain, we get that $D_m*U(H/H_m)\hookrightarrow D$ is the
Ore field of fractions of $D_m*U(H/H_m)$ and $D\hookrightarrow
D_m((H/H_m))$ by Corollary~\ref{coro:fieldoffractionsnilpotentlie}.
Note also that $R_m((H/H_m))\hookrightarrow D_m((H/H_m))$. Hence we
have $K*U(H)\hookrightarrow R_m((H/H_m))\hookrightarrow
D_m((H/H_m))$ and,  by the universal property of the direct limit,
$K*U(H)\hookrightarrow K((H))\hookrightarrow \varinjlim
D_m((H/H_m))$. This implies that the field of fractions of $K*U(H)$
in $K((H))$, which is $K(H)$, and in $\varinjlim D_m((H/H_m))$ must
coincide. Now note that $D$  is the field of fractions of $K*U(H)$
in $\varinjlim D_m((H/H_m))$ because $D$ is the field of fractions
of $K*U(H)$ in $D_m((H/H_m))$ for each $m$. Therefore $K(H)=D$, as
desired.
\end{proof}

By \cite{Hughes}, it is known that if $G$ is an orderable group, $K$
a field and $KG$ a crossed product, then the field of fractions
$K(G)$ inside the Malcev-Neumann series ring $K((G))$ does not
depend on the ordering of $G$. The following theorem should be seen
as an analogous result.

\begin{theo}\label{theo:k(H)doesnotdependontheorder}
The field $K(H)$ does not depend on the RN-series with a Q-basis
chosen, nor on the Q-basis $\mathcal{C}$ chosen nor on the canonical
ordering of $\mathcal{C}_i$ chosen. In fact $K*U(H)\hookrightarrow
K(H)$ and $K*U(H)\hookrightarrow \mathfrak{D}(K*U(H))$ (cf. Proposition~\ref{prop:canonicalfieldoffractions}) are
isomorphic fields of fractions.
\end{theo}

\begin{proof}
First note that the construction of $\mathfrak{D}(K*U(H))$ does not
depend on the RN-series with a Q-basis chosen, nor on the Q-basis
$\mathcal{C}$ nor on the canonical ordering of $\mathcal{C}$, see
\cite[Theorem~2.6.5]{Cohnskew} or \cite{Lichtmanvaluationmethods}.

Let $\{H_i\}_{i=1}^\infty$ be an RN-series with a Q-basis
$\mathcal{C}=\bigcup_{i=1}^\infty \mathcal{C}_i$ and set a canonical
ordering of $\mathcal{C}$.

For each $m\geq 1$, $\bigcup_{i=1}^{m-1}\mathcal{C}_i$ is a set of
elements in $H$ which give a basis of  $H/H_m$. By
Proposition~\ref{prop:canonicalfieldoffractions}(ii), the standard
monomials in $\bigcup_{i=1}^{m-1}\mathcal{C}_i$ are linearly
independent over $\mathfrak{D}(K*U(H_m))$. Hence
$K*U(H)\hookrightarrow K(H)$ and $K*U(H)\hookrightarrow
\mathfrak{D}(K*U(H))$ are isomorphic fields of fractions of $K*U(H)$
by Theorem~\ref{theo:Hughestheoremforliealgebras}.
\end{proof}

The next result should be seen as a weaker version of
\cite[Theorem]{Hughes2}, along the lines of
\cite[Proposition~2.5(3)(ii)]{SanchezfreegroupalgebraMNseries}.

\begin{coro}\label{coro:k(H)ascrossedproduct}
Let $H$ be a  Lie $k$-algebra. Let $K$ be a field containing $k$ as a
central subfield. Consider a crossed product $K*U(H)$.

Suppose that $N$ is an ideal of $H$ such that both $N$ and $H/N$ are
residually nilpotent and they both have RN-series with Q-basis. Then
the natural embedding $K*U(H)\hookrightarrow K(N)(\frac{H}{N})$
gives a field of fractions of $K*U(H)$ isomorphic to
$K*U(H)\hookrightarrow \mathfrak{D}(K*U(H))$.

Moreover, if $H$ is residually nilpotent with an RN-series that has
a Q-basis, then $K*U(H)\hookrightarrow K(H)$ and
$K*U(H)\hookrightarrow K(N)(\frac{H}{N})$ are isomorphic fields of
fractions.
\end{coro}

\begin{proof}
 By
Proposition~\ref{prop:canonicalfieldoffractions}(iii), we have
$\mathfrak{D}(K*U(N))*U(H/N)\subseteq \mathfrak{D}(K*U(H))$. Now
Theorem~\ref{theo:Hughestheoremforliealgebras} and again
Proposition~\ref{prop:canonicalfieldoffractions}, imply that
$\mathfrak{D}\big(\mathfrak{D}(K*U(N))*U(H/N)\big)\cong\mathfrak{D}(K*U(H))$.
By Theorem~\ref{theo:k(H)doesnotdependontheorder}, $K(N)$ and
$\mathfrak{D}(K*U(N))$ are isomorphic fields of fractions of
$K*U(N)$. Hence $\mathfrak{D}(K(N)*U(\frac{H}{N}))\cong
\mathfrak{D}(K*U(H))$. Again by
Theorem~\ref{theo:k(H)doesnotdependontheorder},
$K(N)(\frac{H}{N})\cong \mathfrak{D}(K*U(H))$ as fields of fractions
of $K*U(H)$.

If $K(H)$ exists, then
Theorem~\ref{theo:k(H)doesnotdependontheorder} implies that
$K(H)\cong \mathfrak{D}(K*U(H))\cong K(N)(\frac{H}{N})$.
\end{proof}

We showed in Lemma~\ref{lem:crossedproductfir} that $K*U(H)$, the
crossed product of a field $K$ by $U(H)$ where $H$ is a free Lie
$k$-algebra, is a fir. Thus it has a universal field of fractions.
We are going to prove that $K*U(H)\hookrightarrow K(H)$ and
$K*U(H)\hookrightarrow \mathfrak{D}(K*U(H))$ are both the universal
field of fractions. This result was already known for $U(H)$
\cite[Theorem~1]{Lichtmanuniversalfields}, where  the proof relies on the existence of some specialization (see
\cite[Lemma~3.1]{Lichtmanuniversalfields}). The techniques for the
construction of such specialization do not work for crossed
products. In our proof, the role of
\cite[Lemma~3.1]{Lichtmanuniversalfields} is played by
Proposition~\ref{prop:generalizationProp3.1Lichtman}.

\medskip

\begin{rem} \label{rem:qbasisfree} Let $H$ be a free Lie $k$-algebra. Then $H$ is  graded. Indeed
$H=\bigoplus_{i\geq 1}N_i$ where each $N_i$ is the subspace
generated by the Lie monomials of degree $i$. Then
$H_i=\bigoplus_{j\geq i}N_j$ is the $i$-th term of the lower central
series of $H$. Let $\mathcal{C}_i$ be a basis of $N_i$ for $i\geq
1.$  Therefore we are in the situation of Examples~\ref{ex:qbasis} and we can deduce that $\bigcup _{i\ge 1}\mathcal{C}_i$ is a $Q$-basis of the residually nilpotent algebra $H$. \end{rem}

\begin{theo}\label{theo:LewintheoremforLiealgebras}
Let $H$ be a free Lie $k$-algebra, $K$ a field with $k$ as a central
subfield and consider $K*U(H)$. Then $K*U(H)\hookrightarrow K(H)$
and $K*U(H)\hookrightarrow \mathfrak{D}(K*U(H))$ coincide with the
universal field of fractions of $K*U(H)$.
\end{theo}

\begin{proof}
Denote by $K*U(H)\hookrightarrow E$ the universal field of fractions
of $K*U(H)$. We follow the notation of Remark~\ref{rem:qbasisfree}.

It is known that any subalgebra of a free Lie algebra is a free Lie
algebra. Thus, for each $m\geq 1$, $K*U(H_m)$ is a fir and therefore
it has a universal field of fractions $K*U(H_m)\hookrightarrow
R_{\mathcal{P}_m}$ which, by Lemma~\ref{lem:crossedproductfir}, is a
universal localization at the prime matrix ideal $\mathcal{P}_m$.
Now, by Proposition~\ref{prop:generalizationProp3.1Lichtman},
$K*U(H)\hookrightarrow R_{\mathcal{P}_m}*U(H/H_m)\hookrightarrow E$.
Hence the conditions of
Theorem~\ref{theo:Hughestheoremforliealgebras} are satisfied. Thus
we can deduce that
 $K*U(H)\hookrightarrow E$ and
$K*U(H)\hookrightarrow K(H)$ are isomorphic fields of fractions. By
Theorem~\ref{theo:k(H)doesnotdependontheorder},
$K*U(H)\hookrightarrow \mathfrak{D}(K*U(H))$ is also  isomorphic
to the universal field of fractions of $K*U(H)$.
\end{proof}

For the missing details and definitions  in the next example, the
reader is referred to \cite{Makar-LimanovUmirbaevfreePoissonfields}
and the references therein.

\begin{ex} \label{ex:poisson}
Let $Q=P(x_1,\dotsc,x_n)$ be the free Poison field over $k$ in the
variables $x_1,\dotsc, x_n$ and let $Q^e$ be its universal
enveloping algebra.

In \cite[Theorem~1]{Makar-LimanovUmirbaevfreePoissonfields}, it is
proved that $Q^e$ satisfies the weak algorithm for a certain
filtration of $Q^e$. Thus $Q^e$ is a free ideal ring and, therefore,
it has a universal field of fractions. Although not stated
explicitly, it is also proved in \cite[Proposition~1,
Corollary~1]{Makar-LimanovUmirbaevfreePoissonfields} that $Q^e$ is
in fact a crossed product $K*U(H)$ of a commutative field $K$ over
$U(H)$, the universal enveloping algebra of the free Lie algebra $H$
on $x_1,\dotsc,x_n$. Indeed, by
\cite[Proposition~1]{Makar-LimanovUmirbaevfreePoissonfields}, the
morphism given in \cite[Theorem~5]{ChoOhSkewenveloping} is in fact
an isomorphism by a basis argument, and thus $Q^e$ is a crossed
product as stated. Then, by
Theorem~\ref{theo:LewintheoremforLiealgebras}, $Q^e\hookrightarrow
\mathfrak{D}(Q^e)$ and $Q^e\hookrightarrow K(H)$ are the universal
field of fractions of $Q^e$. We stress that  it cannot be deduced
from the results in \cite{Lichtmanuniversalfields} that  these
embeddings are the universal field of fractions of $Q^e$. \qed
\end{ex}

We remark on passing that if $R$ is an ordered $k$-algebra with
positive cone $P(R)$ (for unexplained terminology see for example
\cite[Section~9.6]{Cohnskew}), and $H$ is a residually nilpotent Lie
$k$-algebra with a Q-basis, then $R((H))$ is an ordered ring for any
crossed product $R*U(H)$. In particular, if $R=K$ is a field,
$K((H))$, $K(H)$ and $\mathfrak{D}(K*U(H))$ are ordered fields.
Indeed, if $\ell\colon R((H))\rightarrow R$ is the least element
map, then $\mathfrak{P}=\{f\in R((H))\mid \ell (f)\in P(R)\}$ is a
positive cone for $R((H))$. Clearly, $\mathfrak{P}\cap
-\mathfrak{P}=\emptyset$ and $\mathfrak{P}\cup -\mathfrak{P}=
R((H))\setminus\{0\}$. Moreover
$\mathfrak{P}\cdot\mathfrak{P}\subseteq \mathfrak{P}$ by
Lemma~\ref{lem:propertiesleastelementmapresidually}(iv), and it is
not difficult to prove that $\mathfrak{P}+\mathfrak{P}\subseteq
\mathfrak{P}$.


\section{Inversion height: the point of view of crossed products of Lie
algebras.}\label{sec:furtherexamples}

Let $R$ be a $k$-algebra with a field of fractions
$\varepsilon\colon R\hookrightarrow D$. Let $H$ be a residually
nilpotent Lie $k$-algebra with an RN-series $\{H_i\}_{i=1}^\infty$
that has a Q-basis $\mathcal{C}=\bigcup_{i=1}^\infty\mathcal{C}_i$.
Consider a crossed product $R*U(H)$ and suppose that it can be
extended to a crossed product structure $D*U(H)$. Then, by
Remark~\ref{rem:extensioncrossedproduct} and
Lemma~\ref{lem:uniqueextensionderivation}, we can consider the
crossed product  $D_\varepsilon(n)*U(H)$ for each $n\geq
0$. Moreover,
$$R*U(H)\hookrightarrow D_\varepsilon(n)*U(H)\hookrightarrow
D_\varepsilon(n+1)*U(H)\hookrightarrow D*U(H).$$ Consider the
embedding $$\iota\colon R*U(H)\hookrightarrow
\mathcal{L}_n=D_\varepsilon(n)((H))\hookrightarrow
\mathcal{L}_{n+1}=D_\varepsilon(n+1)((H))\hookrightarrow D((H))=E.$$

Note that if $f\in D_\varepsilon(n)((H))$, then the least element
map $\ell\colon D((H))\rightarrow D$ is such that $\ell(f)\in
D_\varepsilon (n)$.

With this notation, we can prove an analogous result to
Theorem~\ref{theo:inversionheightseries}.

\begin{theo}\label{theo:inversionheightlieseries} The following hold
true
\begin{enumerate}[\rm(i)]
\item $E_\iota(n)\subseteq \mathcal{L}_n$ for each integer $n\geq
0$.
\item Let $f\in D$. If $\h_\varepsilon(f)=n$, then $\h_\iota(f)=n$.
\item $\h_\iota(R*U(H))\geq \h_\varepsilon(R).$
\end{enumerate}
\end{theo}

\begin{proof}
(i) We proceed by induction on $n$. For $n=0$, the result holds
since $R*U(H)\hookrightarrow R((H))$. Suppose that the result holds
for $n\geq 0$. Let $0\neq f\in E_\iota(n)\subseteq
D_\varepsilon(n)((H))$. Consider the least element map $\ell\colon
D_\varepsilon(n)((H))\rightarrow D_\varepsilon(n)$.

For each $m\geq 1$, set $R_m=D_\varepsilon (n)*U(H_m)$,
$S_m=D_\varepsilon(n+1)*U(H_m)$ and consider $\ell_m\colon
R_m((H/H_m))\rightarrow R_m$.

Set $f_0=f$. There exists a descending chain of ordinals
$$\beta_1=\gamma_1+1>\beta_2=\gamma_2+1>\dotsb>\beta_{r-1}=\gamma_{r-1}+1>\beta_r=0,$$
such that $\beta_i$ is the least ordinal such that $f_{i-1}\in
R_m((N_{\beta_i}))=R_m((N_{\gamma_i}))((z_{\gamma_i};\delta_{u_{\gamma_i}}))$
and $f_i$ is the coefficient of the least element in $\supp f_{i-1}$
as a series in $z_{\gamma_i}$ for $i=1,\dotsc,r$. Moreover,
$\ell(f)=f_{r-1}\in D_\varepsilon(n)$. We prove by recurrence that
$f_{r-p}^{-1}\in S_m((H/H_m))$ for $p=1,\dotsc,r$. For $p=1$, the
result holds because $f_{r-1}^{-1}\in D_\varepsilon(n+1)\subseteq
S_m((H/H_m))$. By definition,
$$f_{r-(p+1)}=f_{r-p}z_{\gamma_{r-p}}^{l_0}+\sum_{l> l_0} c_lz_{\gamma_{r-p}}^l$$
where $c_l\in R_m((N_{\gamma_{r-p}}))$ for all $c_l$. Define
$g_{r-(p+1)}=f_{r-p}z_{\gamma_{r-p}}^{l_0}-f_{r-(p+1)}$. Then
$$f_{r-(p+1)}^{-1}=(f_{r-p}z^{l_0}_{\gamma_{r-p}})^{-1} \sum_{q\geq 0}(g_{r-(p+1)}(f_{r-p}z_{\gamma_{r-p}}^{l_0})^{-1})^q.$$
Now it is not difficult to prove that
$(g_{r-(p+1)}(f_{r-p}z_{\gamma_{r-p}}^{l_0})^{-1})^q\in S_m
((N_{\gamma_{r-p}}))((z_{\gamma_{r-p}};\delta_{u_{\gamma_{r-p}}}))$
for each $q$, and therefore $f_{r-(p+1)}^{-1}\in S_m((H/H_m))$.
Hence $f_0\in S_m((H/H_m))\subseteq D_\varepsilon(n+1)((H))$, as
desired.

(ii) Let $f\in D_\varepsilon(n+1)\setminus D_\varepsilon(n)$. Since
$D_\varepsilon (n+1)\subseteq D$, $\ell(f)=f$. Suppose that $f\in
D_\varepsilon (n)((H))$. Then $f=\ell(f)\in D_\varepsilon (n)$, a
contradiction.

(iii) Follows from (ii).
\end{proof}

Let $I$ be a set of cardinality at least two and let $\{H_i\}_{i\in
I}$ be a set of  Lie $k$-algebras. Set $H$ to be the free product of
such algebras, that is, $H=\coprod_{i\in I} H_i$. Consider the
direct sum $\bigoplus_{i\in I} H_i$. For each $i\in I$, let
$\pi_i\colon H_i\hookrightarrow \bigoplus_{i\in I} H_i$ be the
canonical inclusion. Let $\pi\colon \coprod_{i\in I}H_i\rightarrow
\bigoplus_{i\in I}H_i$ be the unique morphism of Lie $k$-algebras
such that $\pi_{\mid_{H_i}}=\pi_i$. The subalgebra $\ker \pi$ is
called the \emph{cartesian subalgebra} of the free product $H$.

By $[x,y]_n$, we denote the product $[[\dotso[[x,y],y],\dotso],y]$
with $n$ factors $y$.


\begin{coro}\label{coro:freeenvelopinginfiniteinversionheight}
Let $I$ be a set of cardinality at least two and let $\{H_i\}_{i\in
I}$ be a set of  nilpotent Lie $k$-algebras.  Set $H=\coprod_{i\in
I} H_i$. Let $U(H)$ be the universal enveloping algebra of $H$ and
consider the embedding $\iota\colon U(H)\hookrightarrow
\mathfrak{D}(K*U(H))$. Then $\h_\iota(U(H))=\infty$. Indeed, let
$x\in H_i\setminus\{0\}$ and $y\in H_j\setminus\{0\}$ with $i\neq
j$. If $f$ is any entry of the inverse of the $n\times n$ matrix
$$A_n=\left(\begin{array}{cccc}
[[x,y],x] & [[x,y],x]_2 & \dots & [[x,y],x]_n \\
{} [[x,y]_2,x] & [[x,y]_2,x]_2 & \dots & [[x,y]_2,x]_n \\
\dots & \dots & \ddots & \dots\\
{} [[x,y]_n,x] & [[x,y]_n,x]_2 & \dots & [[x,y]_n,x]_n
\end{array}\right)$$
then $\h_\iota(f)=n$.

In particular, if $X$ is a set of cardinality at least two and
$\freealgebra kX$ is the free  $k$-algebra on $X$, then the
universal field of fractions $\iota\colon\freealgebra kX
\hookrightarrow  F$ is of infinite inversion height. Indeed, let
$x,y\in X$ be different elements. If $f$ is any entry of the inverse
of the $n\times n$ matrix,
$$A_n=\left(\begin{array}{cccc}
[[x,y],x] & [[x,y],x]_2 & \dots & [[x,y],x]_n \\
{} [[x,y]_2,x] & [[x,y]_2,x]_2 & \dots & [[x,y]_2,x]_n \\
\dots & \dots & \ddots & \dots\\
{} [[x,y]_n,x] & [[x,y]_n,x]_2 & \dots & [[x,y]_n,x]_n
\end{array}\right)$$
then $\h_\iota(f)=n$.
\end{coro}

\begin{proof}
Let $N$ be the cartesian subalgebra of $H$. By \cite[Theorem
4.10.5]{BokutKukinbook}, $N$ is a free Lie $k$-algebra on an
infinite set $Y$, and thus $U(N)$ is a free $k$-algebra on $Y$.
Moreover, it is not difficult to see that $H/N\cong \bigoplus_{i\in
I} H_i$ is a residually nilpotent Lie $k$-algebra with an RN-series
that has a Q-basis. By Corollary~\ref{coro:k(H)ascrossedproduct},
$U(H)\hookrightarrow \mathfrak{D}(K*U(H))$ can be seen as
$U(H)\hookrightarrow k(N)(\frac{H}{N})\hookrightarrow
k(N)((\frac{H}{N}))$. By
Theorem~\ref{theo:reutenauerinversionheight} and
Theorem~\ref{theo:LewintheoremforLiealgebras}, $\varepsilon\colon
U(N)\hookrightarrow K(N)$ is of infinite inversion height. By
Theorem~\ref{theo:inversionheightlieseries}, $h_\iota
(U(H))=\infty$.

Moreover, using \cite[Section~4.10]{BokutKukinbook},  $Y$ can be
chosen to contain the elements $[[x,y]_i,x]_j$. By
Theorem~\ref{theo:reutenauerinversionheight}, for each entry $f$ of
$A_n^{-1}$, $\h_{\varepsilon}(f)=n$. Applying
Theorem~\ref{theo:inversionheightlieseries}, we obtain that
$\h_\iota(f)=n$.

When $H$ is the free Lie algebra on a set $X$, put $I=X$.  Then $H$
is the free product of the abelian (and hence nilpotent) Lie
$k$-algebras generated by each $x\in X$. Now apply the foregoing,
and note that $\mathfrak{D}(U(H))$ is the universal field of
fractions of $U(H)$, by
Theorem~\ref{theo:LewintheoremforLiealgebras}.
\end{proof}

We remark that the statement of
Corollary~\ref{coro:freeenvelopinginfiniteinversionheight} works for
any set $\{H_i\}_{i\in I}$ of residually nilpotent Lie $k$-algebras
with a Q-basis because they induce a natural  RN-series with a
Q-basis in $\bigoplus_{i\in I}H_i$. Also, it is known that the free
product of residually nilpotent Lie algebras is a residually
nilpotent Lie algebra, see for example \cite[p.175]{BokutKukinbook}.
On the other hand, we do not know whether there exists an RN-series
of the free product with a Q-basis.

Note that  by choosing different elements (or changing the basis) of
$N$, other elements of prescribed inversion height $n$ can can be
found.

Another way of obtaining the second part of
Corollary~\ref{coro:freeenvelopinginfiniteinversionheight} is the
following. By \cite{Baumslagfreeliealgebras}, if $N\neq H$ is an
ideal of the free (not commutative) Lie algebra $H$, then $N$ is a
free Lie algebra not finitely generated. Thus, choosing $N$ such
that $H/N$ is nilpotent, we get  elements of inversion height
$n$ for any $n^2$ different free generators of $N$.

\section*{Acknowledgments}

Both authors are grateful to Bill Chin for pointing out suitable references on  Lie algebra crossed products.

The second author would like to thank Alexander Lichtman for
interesting conversations on the papers
\cite{Lichtmanvaluationmethods}, \cite{Lichtmanuniversalfields} and
related topics. He also thanks the Mathematics Department at the
University of Wisconsin Parkside, where these conversations took
place, for its kind hospitality during his visit.

\bibliographystyle{amsplain}
\bibliography{grupitosbuenos}

\end{document}